\documentclass[10pt,reqno]{amsart}
\usepackage{amssymb, amsmath, graphicx, amsfonts, euscript, mathrsfs}
\usepackage{color}
\usepackage{tikz}
\usepackage{dsfont}

\newtheorem{thm}{Theorem}[section]

\newtheorem{prop}[thm]{Proposition}

\newtheorem{lemma}[thm]{Lemma}
\theoremstyle{remark}
\newtheorem{remark}{Remark}
\newtheorem{example}[thm]{Example}

\newcommand{\e}{\epsilon}
\newcommand{\p}{\partial}
\newcommand{\g}{\Gamma}
\newcommand{\A}{\alpha}

\newcommand{\na}{\nabla}
\newcommand{\la}{\lambda}

\newcommand{\R}{\mathbb{R}}
\newcommand{\Z}{\mathbb{Z}}
\newcommand{\N}{\mathbb{N}}

\newcommand{\F}{\mathcal{F}}

\newcommand{\PP}{\mathcal{P}}

\newcommand{\HH}{\mathbb{H}}
\newcommand{\les}{\lesssim}

\newcommand{\E}{\mathcal{E}}
\newcommand{\U}{\mathcal{U}}

\newcommand{\h}{\mathcal{H}}

\newcommand{\ind}[1]{\mathds{1}_{{#1}}}

\def\beq{\begin{equation}}
\def\eeq{\end{equation}}
\numberwithin{equation}{section}
\begin{document}
\title[ Variational Estimates of Stein-Wainger type operators]{Sharp variational estimates of Stein-Wainger type operators}

\author[R. Wan]{ Renhui Wan}

\address{Ministry of Education Key Laboratory of NSLSCS, School of Mathematical Sciences and Institute of Mathematical Sciences, Nanjing Normal University, Nanjing 210023, People's Republic of China}

\email{wrh@njnu.edu.cn}


\vskip .2in
\begin{abstract}
For any integer $n \geq 2$,  we establish  $L^p(\R^n)$ inequalities for the $r$-variations of Stein-Wainger type oscillatory integral operators with general phase functions.
These inequalities   closely  related to Carleson's theorem are sharp, up to endpoints.
 In particular,
 when the phase function is chosen as $|t|^\A$ with $\A\in (0,1)$,
 our results provide an affirmative   answer to a question posed in Guo-Roos-Yung (Anal. PDE, 2020).
 Furthermore,  we obtain the restricted weak type estimates for
endpoints  in the specific case of homogeneous phase functions.

\end{abstract}

\maketitle

\section{Introduction}
\label{intrs1}
\subsection{Motivation and main results}\label{Mr}
Let $n\geq 2$, and consider a function $\g: \R \to \R$ that is smooth on $\R^+$ and right-continuous at the origin. For each $x \in \R^n$ and $u \in \R$, we define a modulated  singular integral as follows:
$$H^{(u)} f(x):={\rm p.v.}\int_{\R^n} f(x-t)e^{iu\g(|t|)}K(t)dt,$$
where  $K(t)$ is a homogeneous Calder\'{o}n-Zygmund
kernel, in the sense  that
\beq\label{v1}
K(t)=\frac{\Omega(t)}{|t|^n}
\eeq
for some function $\Omega$ which is smooth on $\R^n\setminus\{0\}$, homogeneous of degree 0,\footnote{The assumption that $K$ is homogeneous is not strictly necessary; its inclusion is aimed at streamlining the proof of main results and subsequently improving the presentation of this paper.}  and possesses the property of mean value zero, that is,  $\int_{\mathbb{S}^{n-1}}\Omega(x)d\sigma(x) = 0$, where $\sigma$ denotes the surface measure on  $\mathbb{S}^{n-1}$.   Henceforth $f$ will always be a Schwartz function on $\R^n$, and the objective is to establish a priori $L^p$ inequalities for  $r$-variations  of  $\{H^{(u)}f\}_{u\in\R}$ for all such $f$. Here,   the $r$-variation is defined   as follows:
For $r\in [1,\infty)$ we define
the $r$-variation  $V_r$ of a complex-valued function $\R\times \R^n \ni (u,x) \to {\bf a}_u(x)$
 by setting
\beq\label{dfvr}
V_r({\bf a}_u(x): {u \in \mathcal{K}}):=
 \sup_{J\in\N} \sup_{\substack{u_{0} <  \dotsb < u_{J}\\ \{u_{j}\}\subset\mathcal{K}}}
\big(\sum_{j=1}^{J}  |{\bf a}_{u_{j}}(x)-{\bf a}_{u_{j-1}}(x)|^{r} \big)^{1/r}
\eeq
where $\mathcal{K}$ is a subset of $\R$.
  The seminorm $V_r$  serves as a crucial tool in addressing pointwise convergence issues. If the estimate
 $$V_r\big({\bf a}_u(x):u\in \R\big)<\infty$$
 holds for some $r\in [1,\infty)$ and $x\in \R^n$, then,  the limits
 $\lim_{u\to u_0}{\bf a}_u(x)$ ($u_0$ is finite), $\lim_{u\to  -\infty}{\bf a}_u(x)$ and $\lim_{u\to \infty}{\bf a}_u(x)$ exist.
 Consequently, there is no need to establish pointwise convergence over a dense class, which is often a challenging problem (see, for example, Mirek--Stein--Trojan \cite[Page 668]{MST17}). In addition, the seminorm $V_r$ governs the supremum norm as follows: For any $u_0\in \R$, we can obtain  the pointwise estimate
 $$\sup_{u\in\R}|{\bf a}_u(x)|\le |{\bf a}_{u_0}(x)|+V_r\big({\bf a}_u(x):u\in \R\big).$$
A norm on the space
 $V^r(\mathcal{K})$  is given by
$\|\mathbf{a}_u\|_{V_r(u\in \mathcal{K})}:=\|\mathbf{a}_u\|_{L^\infty(u\in \mathcal{K})}+V_r\big(\mathbf{a}_u:u\in \mathcal{K}\big)$.
We also introduce the related  $\la$-jump $(\la>0)$ function of the family
$\{\mathbf{a}_u(x)\}_{u\in \mathcal{K}}$,  denoted by $N_\la\big(\mathbf{a}_u(x):u\in \mathcal{K}\big)$. It is
defined as the supremum
of all positive integers $J$ for which there is a strictly increasing sequence $s_1<t_1<s_2<t_2<\cdots<s_J<t_J$, all in $\mathcal{K}$,
such that
$|\mathbf{a}_{t_j}(x)-\mathbf{a}_{s_j}(x)|>\la$
for all $j=1,\ldots,J$.
Introducing the jump inequality is convenient for estimating  long variations in the proofs of our main results.
For more details on the relationship between variational inequalities and jump inequalities, we refer to \cite[Pages 554-558]{FITW20}  and \cite{JSW08,MSZ1}.

For the case of $r=\infty$, we obtain  the  associated  $V_\infty$  seminorm. Since this seminorm  of  the family $\{H^{(u)} f\}_{u\in\R}$   is closely
connected to the  maximal function
\beq\label{nan1}
\sup_{u\in\R}|H^{(u)} f|,
\eeq
 we begin by presenting some pertinent results concerning this maximal function.
The maximal operator defined by (\ref{nan1}) with $\g(|t|)=|t|^\A$, where $\A>1$, was introduced by Stein and Wainger \cite{SW21}. It serves as a variant of the Carleson operator studied in \cite{Car96,Fe73,LT20}. Furthermore, this operator with an even $\A$ can also be viewed as a special polynomial Carleson operator, a central focus in a conjecture proposed by Stein. Recently, this conjecture was resolved by V. Lie \cite{LV20} in the one-dimensional case  and by P. Zorin-Kranich \cite{ZK21} in higher dimensions.
  The aforementioned results are related to the continuous setting, and there are also important research efforts in the realm of discrete analogues.
   Specifically, the discrete analogue of (\ref{nan1}) is defined as follows:
  \beq\label{d1001}
  \sup_{u\in\R}\big|\sum_{t\in \Z^n\setminus \{0\}} f(x-t)e^{iu\g(|t|)}K(t)\big|,\ \ \ x\in \Z^n,
  \eeq
  where $K(t)$ is given by (\ref{v1}).
  Krause and Roos in their seminal works \cite{KR22,KR23} established $\ell^p(\Z^n)$ estimates for (\ref{d1001})  when $\g(|t|)=|t|^{2d}$ with positive integers $d$.
Recently, by Weyl's equidistribution theory, Krause \cite{K23} extended phase functions studied in \cite{KR22,KR23} to  generic polynomials  without linear terms.
Their proofs reveal that addressing the case of more general phase functions introduces additional complexities. See \cite{WC26} for related work concerning the variational inequality in the scale parameter, as opposed to the modulation parameter $u$ considered here.

For $r\in (2,\infty)$, L\'{e}pingle's inequality on the $r$-variation of martingales is particularly noteworthy, as discussed in \cite{L76,PX88}. Besides, various works have explored the $r$-variation of significant operators in ergodic theory and harmonic analysis; for instance, see \cite{JSW08,BMSW18,BMSW19,BORSS22,Bour89,CJRW00,CJRW03,MSZ20,MSZ202} and references therein. In fact, $L^p$ inequalities for  $r$-variations may be used to establish  the discrete analogues of maximal estimate; see, for example,  \cite{GRY20,K18}.
Subsequently, we outline some works that specifically address variation-norm estimates, closely connected to our primary results.
In the one-dimensional scenario, with the phase function $\g(|t|)$ replaced by $t$, the $r$-variation estimate associated with the Carleson operator was investigated  by Oberlin, Seeger, Tao, Thiele, and Wright \cite{OSTTW12}. More recently, for $n \ge 1$, $p \in (1,\infty)$, $r \in (2,\infty)$, and $r > p'/n$, and with $\g(|t|)=|t|^\A$ where $\A>1$, Guo, Roos, and Yung \cite{GRY20} obtained sharp $L^p$ estimates for the $r$-variation of the operators $\{H^{(u)}\}_{u\in\R}$, up to  endpoints $r=p'/n$. 
Their work presents an insightful synthesis, skillfully harnessing square function estimates for fractional Schr\"{o}dinger operators and deploying the  $\ell^p$
  decoupling inequality (see, e.g., \cite{BD15,W00}) to establish local smoothing estimates. 
Crucially, the estimate for the maximal function (\ref{nan1}), established in \cite{SW21}, served as a fundamental tool, akin to a black box, in demonstrating the main results in \cite{GRY20}.
However, no information is provided for the endpoints $r=p'/n$, and the authors in \cite{GRY20} posed a natural question regarding  what happens when $\A$ falls within the range of $(0,1)$ for $n\ge 2$. As of now, these questions remain open. Motivated by these considerations, we focus on the following questions:
\begin{itemize}
\item
 {\bf Question} 1: For $n \ge 2$, do analogous $L^p$ inequalities hold for $r$-variations of the operators $\{H^{(u)}\}_{u\in\R}$, encompassing generic phase functions, including the case of $|t|^\A$ with $\A\in (0,1)$?
\smallskip

\item
 {\bf Question} 2: What happens for the endpoints $r=p'/n$?
\end{itemize}

To state our main results, we will  introduce the set $\U$ (see \cite{LY22,LSY21,LYu22,LV24} for some important  related variants), which consists of all functions $\g$ satisfying $\g\in C([0,\infty))\cap C^N(\R^+)$ with $N\in\N$ sufficiently large, $\g(0)=0$, $\g'>0$ on $\R^+$, and complying with the following conditions:
\begin{align}
(i)\  \text{There are positive constants}~ \{\mathfrak{C}_{1,j}\}_{j=1}^2 \text{ such that}~   \left|\frac{s^j\g^{(j)}(s)}{\g(s)}\right|\ge \mathfrak{C}_{1,j} \text{ for all } s\in\mathbb{R}^+; \label{curv}\\
(ii)\  \text{There are positive constants}~ \{\mathfrak{C}_{2,j}\}_{j=1}^N \text{ such that}~ \left|\frac{s^j\g^{(j)}(s)}{\g(s)}\right|\le \mathfrak{C}_{2,j} \text{ for all } s\in\mathbb{R}^+. \label{smooth}
\end{align}

One of the main results of this paper, addressing question 1, is   the following theorem.
\begin{thm}\label{endpoint}
   Let $n\ge 2$ and  $\g\in \U$. If $p\in(1,\infty)$, $r\in(2,\infty)$ and $r>p'/n$, then we have
\beq\label{var}
\|V_r\big(H^{(u)} f: u\in\R\big)\|_{p}
\les \|f\|_{p}.
\eeq
Moreover, there exist Schwartz functions $f$ such that  (\ref{var})  fails whenever $r<p'/n$.
\end{thm}
Furthermore, we
address question 2
 for  $n\ge2$ and for phase functions  satisfying an additional homogeneous condition\footnote{Indeed, we are handling a slightly broader case, requiring only the satisfaction of condition (\ref{homo}).}
  given by
 \beq\label{homo}
\g(st)\sim \g(s)\g(t), \quad {\rm for\ all}\ s>0,\ t>0.
\eeq
Specifically, the result corresponding to the endpoints of Theorem \ref{endpoint}  is as follows:
\begin{thm}\label{endpoint2}
   Let $n\ge 2$,    $\g\in \U$, and suppose that $\g$ satisfies  (\ref{homo}). If $p'\in(2n,\infty)$,
\beq\label{var2}
\big\|V_{p'/n}\big(H^{(u)} f: u\in\R\big)\big\|_{L^{p,\infty}}
\les \|f\|_{L^{p,1}}.
\eeq
\end{thm}
\vskip.1in
Let us make some comments on our main results.
\begin{itemize}
\item
For $n\geq 2$ and $\g\in\mathcal{U}$, we can establish the following  jump inequality:
\begin{equation}\label{jump100}
\sup_{\lambda>0} \big\|\lambda \sqrt{N_\lambda(H^{(u)} f: u\in\mathbb{R})}\big\|_p \lesssim \|f\|_p,
\quad {2n}/{(2n-1)} < p < \infty
\end{equation}
(see $Remark$ \ref{amz} in Subsection \ref{sub6.7} below for further details).

 \item The set $\U$ encompasses numerous interesting examples, as illustrated in $Example$ \ref{ex1} below. Particularly, our result in Theorem \ref{endpoint} coincide with that in \cite{GRY20} when we choose $\g(|s|)$ as a homogeneous function $|s|^\A$ with $\A>1$. Furthermore, in the case of $\g(|s|)=|s|^\A$ with $0<\A<1$, Theorem \ref{endpoint} gives an affirmative   answer to a question raised in \cite{GRY20} (see page 1458 there).\footnote{Through suitable adjustments to the techniques outlined in this paper, we can also handle the situation where $\g(|s|) = |s|^\A$ with $\A<0$, a scenario not encompassed by the set $\U$.}

\item
The method employed to demonstrate Theorem \ref{endpoint} appears to face difficulties when applied to the case of $n=1$. In contrast to homogeneous phase functions for $n=1$, the generic  phase functions, especially in  the non-homogeneous case, may require the establishment of local smoothing estimates with variable coefficients (see $Remark$ \ref{r1d} in Subsection \ref{sub6.7}). At present, establishing this type of local smoothing estimates without imposing additional conditions is  challenging.

 \item
Theorem \ref{endpoint2} represents the initial occurrence of an endpoint result, albeit with the requirement that the phase functions satisfy  (\ref{homo}). We conjecture that a similar endpoint result holds for $n=1$ at least in the case of $\g(|s|)=|s|^\A$ with $\A>1$. Nevertheless, the methodology employed in this context, which heavily depends on an interpolation trick by Bourgain, is not applicable in the case of $n=1$.

 \end{itemize}

\begin{example}\label{ex1}

We enumerate some intriguing instances from the set $\U$. Due to the characteristics of the phase function, we just present their expressions within the domain $\mathbb{R}^+$.
\begin{itemize}

\item $\g(s)=s^\A$ with $\A>0$ and $\A\neq 1$;

\item  $\g(s)=\sum_{i=1}^ks^{\A_i}\ln^{\beta_i}(e+s)$, where $k\in \N$, $\A_i>1$ and $\beta_i\ge 0$ for all $ i=1,2,\cdots, k$;

\item $\g(s)=\sum_{i=1}^ks^{\A_i}$, where $k\in \N$, and $\A_i\in (0,1)$ for all $i=1,2,\cdots, k$;

\item  $\g(s)=\sum_{i=1}^k\big((1+s^2)^{\A_i}-1\big)$, where $k\in \N$, and  $\A_i>1$ for all $i=1,2,\cdots, k$.

    \end{itemize}
\end{example}
\begin{figure}\label{fig}
\begin{tikzpicture}[xscale=7, yscale=7]\label{dd}
\path [fill=white] (0,0)--(0,1/2)--(3/4,1/2)--(1,0)--(0,0);
\draw [thick, <->] (0,0.6)--(0,0)--(1.2,0);
\node [below right] at (1.1,0) {${1}/{p}$};
\node [left] at (0,0.6) {${1}/{r}$};
\draw [ thick] (0,1/2)--(3/4,1/2);
\draw [ thick] (3/4,1/2)--(1,0);
\draw [fill] (1/4,0) circle [radius=0.2pt];
\node [below] at (1/4,0) {$D$};
\node [above] at (1/4,1/2) {$B$};
\draw [ thick] (1/4,0)--(1/4,1/2);
\node [below left] at (3/4,1/8) {$E$};
\draw [ thick] (3/4,1/8)--(3/4,0.5);
\draw [  thick] (3/4,1/8) --(1,0);
\draw [fill] (1,0) circle [radius=0.2pt];
\draw [fill] (0,1/2) circle [radius=0.2pt];
\draw [fill] (3/4,1/8) circle [radius=0.2pt];
\draw [fill] (1/4,1/2) circle [radius=0.2pt];
\draw [fill] (3/4,1/2) circle [radius=0.2pt];
\draw [fill] (0,0) circle [radius=0.2pt];
\node [below ] at (1,0) {$F(1,0)$};
\node [below] at (0,0) {$O(0,0)$};

\node [right] at (1,3/8) {${\rm Endpoints}$};
\node [right] at (1,3/10) {$1/r=n/p'$};
\draw [thick, ->] (7/8,1/4)--(1,3/9);


\node [left] at (0,1/2) {$A(0,1/2)$};

\node [above] at (3/4,1/2) {$C$};

\end{tikzpicture}
\caption{
$B(\frac{n-1}{2n+2},\frac{1}{2}),~
C(\frac{2n-1}{2n},\frac{1}{2}), ~D(\frac{n-1}{2n+2},0),~ ~E(\frac{2n-1}{2n},\frac{1}{2n}). 
$
}
\label{fig:1}
\end{figure}
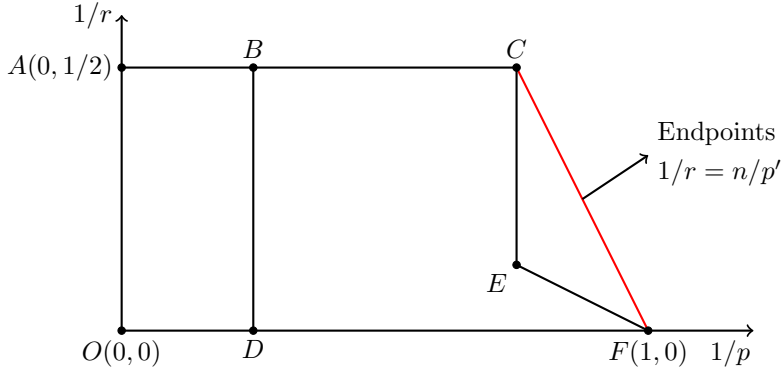
\subsection{Strategies and novelties} \label{meth}

Let ${\bf G}$ denote the enclosed area delineated by the trapezoid AOFC in Figure \ref{fig:1}. We define ${\bf G}_1$ as the interior of  ${\bf G}$. Consequently, the domain for $(p, r)$ in Theorem \ref{endpoint} corresponds to ${\bf G}_1$; and
in the context of Theorem \ref{endpoint2}, the domain for $(p, r)$ is restricted  to $\overline{\mathrm CF}$ with the exclusion of two endpoints $C$ and $F$.
The initial step in proving both (\ref{var}) in Theorem \ref{endpoint} and (\ref{var2}) in Theorem \ref{endpoint2} involves reducing the objective to the long and short variation-norm estimates.
  \begin{itemize}

  \item

Strategies for proving (\ref{var}) in  Theorem \ref{endpoint}: The long and short variation-norm estimates are formulated in (\ref{ka1}) in Proposition \ref{long0} and (\ref{ka2}) in Proposition \ref{short}, respectively. Rather than directly demonstrating (\ref{ka1}), our approach involves establishing its associated jump inequality, as indicated in (\ref{ka}) below. This jump inequality is derived by leveraging a known result from \cite{JSW08}, making essential modifications to the $\lambda$-jump function, and deriving $L^p$ decay estimates for a square function.
In contrast to the proof of the long variation-norm estimate, the proof of the short variation-norm estimate (\ref{ka2}) is significantly more intricate. We present a stronger estimate within a Besov space (see Subsection \ref{othef} below) instead of directly demonstrating (\ref{ka2}). Furthermore, we establish the desired estimates for two subranges of $(p,r)$ using different techniques and achieve the final goal through interpolation. Precisely, by referencing Figure \ref{fig:1},
the aforementioned   two subranges correspond to the interior of the rectangle AODB  and the interior of the triangle CEF.
\smallskip

\item

Strategies for establishing the necessity of $r\ge p'/n$
 in  Theorem \ref{endpoint}: As in the arguments presented in \cite{GRY20} for  special homogeneous phase functions, the arguments  for generic functions  heavily relies on the stationary phase theorem as well. However, the generality of the phase function makes the proof more intricate.
More precisely, we will address two aspects. The first involves the specific properties of the phase function on which the lower bound ultimately depends. The second focuses on ensuring that the phase functions of two oscillatory integrals have  critical points when applying the stationary phase theorem (see Section \ref{sharp}).

\smallskip

 \item
Strategies for proving (\ref{var2}) in Theorem \ref{endpoint2}: This theorem delivers the endpoint result, specifically providing a restricted weak-type variation-norm estimate in Theorem \ref{endpoint2}. The desired long variation-norm estimate is established by Proposition \ref{long0}, while the desired short variation-norm estimate is derived by combining partial arguments from the proof of Proposition \ref{short}, the additional homogeneous condition (\ref{homo}) for phase functions, and an interpolation trick from Bourgain.
\end{itemize}

We next discuss the novelties and challenges encountered in the proof of our main results. Specifically, the key novelties and difficulties primarily arise in establishing the short variation-norm estimate.

\begin{itemize}

\item
In dealing with generic phase functions $\g$ in this context, it's important to note that these functions may not necessarily exhibit the  property $\g(st)=\g(s)\g(t)$. As a result, the proofs for the corresponding long and short variation-norm estimates become more intricate compared to the case of homogeneous phase functions. Similar complexities have been observed in the studies involving associated maximal operators (see, for example, \cite{LY22,LSY21,LYu22}, which inspired us to consider more general phase functions), analogous to the endpoint case discussed here with $r=\infty$. However, the scenario we are addressing encompasses the range $r\in(2,\infty)$, introducing an additional layer of complexity to the problem. This is because we now have to establish corresponding estimates with faster decay.

\smallskip

 \item

The generality of  phase functions necessitates the establishment of variable coefficient square function estimates, which presents considerable complexity.  Moreover, these estimates may impose additional constraints on the phase function $\g$. To address this difficulty,  we adopt an alternative approach incorporating the method of stationary phase, a square function estimate derived from Seeger's arguments (see \cite{Se88}), and Stein-Tomas restriction estimates. In fact,  we only employ certain conditions related to  non-zero Gaussian curvature, without the necessity to distinguish scenarios involving both positive and negative curvature simultaneously.

\smallskip

\item
We will make use of a crucial observation. More precisely, for $(p,r)$ in the interior of the triangle CEF in Figure \ref{fig:1}, the desired estimate (\ref{Go2}) from Proposition \ref{short2} can be obtained through basic interpolations and  pointwise estimates provided by (\ref{decy}) for an oscillatory integral with non-vanishing Gaussian curvature. Consequently, there is no need to rely on the maximal estimate as a black box, and we can establish the endpoint result in Theorem \ref{endpoint2}.

\end{itemize}

\subsection{Organization of the paper.} \label{organ}
In Section \ref{pres2}, we establish several valuable estimates applicable to the function $\g$ within $\mathcal{U}$, provide comparisons between the space $V_r$ and certain Besov spaces, and present a fundamental decomposition for $H^{(u)}f$. In Section \ref{redut1}, we reduce  the proof of (\ref{var})  in   Theorem \ref{endpoint} to proving Propositions \ref{long0} and \ref{short}, which respectively provide
 the desired long and short variation-norm estimates. In Section \ref{long},
  we  demonstrate Proposition \ref{long0}. In Section \ref{s5p}, we reduce the proof of Proposition \ref{short} to proving Lemmas \ref{l1} and \ref{l2},
  and then prove Lemma \ref{l1}.
In Section \ref{Pftheorem1},  we prove Lemma \ref{l2} by establishing Propositions \ref{short2} and \ref{short1}. In Section \ref{sargu}, we give  the proof of Lemma \ref{6l} which plays a crucial role in proving Lemma \ref{l2}.
  In Section \ref{sharp}, we show
  the sharpness of (\ref{var}) as presented in Theorem \ref{endpoint}.
  In Section \ref{weakht}, we prove
  Theorem \ref{endpoint2}, which presents
  an
  endpoint result for  phase functions satisfying the condition (\ref{homo}). 

\subsection{Notation.} \label{Not}
   For any two quantities $x,y$ we will write $x\lesssim y$ and $y\gtrsim x$ to denote   $x\le Cy$ for some absolute constant $C$. If we need the implied constant $C$
to depend on additional parameters,  we will denote this by subscripts. For example $x\les_\rho y$ denotes
$x\les C_\rho y$ for some $C_\rho$ depending on $\rho$.
 If both $x\lesssim y$ and $x\gtrsim y$ hold, we use $x\sim y$. To abbreviate the notation we will sometimes permit the implied constant to depend on certain fixed parameters when the issue of uniformity with respect to such parameters is not of relevance.
  We use
$\mathcal{F}\{f\}$ or $\widehat{f}$ to represent   the Fourier transform of a function $f$, and $\mathcal{F}^{-1}\{g\}$ or $\check{g}$ is the Fourier inverse transform of $g$. More precisely,
$$\mathcal{F}\{f\}(\xi)=\hat{f}(\xi)=\int_{\R^n} f(x)e^{-i\xi\cdot x} dx,\ \  \mathcal{F}^{-1}\{g\}(x)
=\check{g}(x)=(2\pi)^{-n}\int_{\R^n} g(\xi)e^{-i\xi\cdot x} d\xi.$$
For convenience, we omit the constant $(2\pi)^{-n}$ in  the  Fourier inverse transform  throughout  this paper. For $S\subset \Z$, we utilize
$\|a_k\|_{\ell^r(k\in S)}:=(\sum_{k\in S}|a_k|^r)^{1/r}$ if $r<\infty$ and $\|a_k\|_{\ell^\infty(k\in S)}:=\sup_{k\in S}|a_k|$;
 we write $\|a_k\|_{\ell^r_k}:=\|a_k\|_{\ell^r(k\in \Z)}$ in particular.
 Moreover, $\ind{E}$ denotes the indicator function of a set $E$. If $P$ is a statement, we use $\ind{P}$ to denote the indicator function, equal to 1 if $P$ is true and 0 if $P$ is false

Besides, throughout   this paper we fix a cutoff function $\psi:\R\to [0,1]$,
which   is supported  in $\pm[1/2,2]$,
and set $\psi_l(t):=\psi(2^{-l}|t|)$ for any $l\in\Z$ such that  the partition of unity $\sum_{l\in\Z}\psi_l(t)=1$ holds for all $t\in \R^n\setminus \{0\}$. For each $j\in\Z$, we denote by $P_j$ the projection on $\R^n$ defined by
$\widehat{P_jf}(\xi)=\psi_j(\xi)\widehat{f}(\xi)$; we use $\tilde{P}_{j}$ which may vary line by line  to denote a variant of the Littlewood-Paley operator ${P}_{j}$.
\section{Preliminaries}
\label{pres2}
\subsection{Useful estimates on  $\g$}
\label{Fursub1}
In this subsection, we  investigate some further estimates with respect to  the function $\g\in\U$. (One can see \cite{LY22,LSY21,LYu22} for  similar estimates.)  Initially, we can deduce there are two positive constants, $\A_0$ and $\A_1$, such that for all $l\in[0,\infty)$ and all $k\in\R$,
\begin{equation}\label{p8}
2^{\A_0 l} \le \frac{\g(2^{k+l})}{\g(2^{k})} \le 2^{\A_1 l}.
\end{equation}
Indeed, by leveraging the conditions (\ref{curv}) and (\ref{smooth}) with $j=1$ and $\g'>0$ on $\R^+$, we observe the inequality  $\mathfrak{C}_{1,1} s^{-1}\le (\ln\g)'(s)\le \mathfrak{C}_{2,1}s^{-1}$. This   immediately yields (\ref{p8}) when integrating this inequality over the interval $[2^k,2^{k+l}]$.
Moreover, by utilizing (\ref{curv}) with $j=2$, we can deduce that one of the following statements holds:
\vspace{2pt}

(S1) The function $\g$ is convex on $\R^+$,  meaning $\g''(s)>0$ for all $s\in \R^+$;
\vspace{2pt}

 (S2) The function $\g$ is concave on $\R^+$,  meaning $\g''(s)<0$ for all $s\in \R^+$.
 \vspace{2pt}

  Similarly to the preceding  proof of (\ref{p8}), we can also obtain  that there are two positive constants, $\A_0'$ and  $\A_1'$, such that for all $l\in[0,\infty)$ and all $k\in\R$,
\begin{align}
2^{\A_0'l} \le&~ \frac{\g'(2^{k+l})}{\g'(2^{k})}\le~ 2^{\A_1'l}\ \ \ \ \hspace{1.5pt} {\rm if\ (S1)\ holds,\ and} \label{p80}\\
2^{-\A_1'l} \le&~ \frac{\g'(2^{k+l})}{\g'(2^{k})}\le~ 2^{-\A_0'l} \ \ \ {\rm if\ (S2)\ holds}. \label{p81}
\end{align}

Below, we  establish estimates related to the inverse function of $\g'$.
\begin{lemma}\label{lpro}
Let   $\gamma:=(\g')^{-1}$ where $\g\in \U$. Then, there exist two positive constants $\beta_0$ and $\beta_1$ such that for all $l\in[0,\infty)$ and $k\in\R$,
\begin{align}
2^{\beta_0l}\le&~ \frac{\gamma(2^{k+l})}{ \gamma(2^k)}\le~2^{\beta_1l} \ \  \ \hspace{5pt} {\rm if\ (S1)\ holds,\ and}\label{1121}\\
2^{-\beta_1 l}\le&~ \frac{\gamma(2^{k+l})}{ \gamma(2^k)}\le~2^{-\beta_0l} \ \  \ {\rm if\ (S2)\ holds}.
\label{1122}
\end{align}
\end{lemma}
\begin{proof}
  For every $l\ge0$ and  $k\in\R$, we define two real numbers
  $w_{k+l}$ and $w_k$   by $2^{w_i}:=\gamma(2^i)>0$,
$i=k+l,k$. Then
${\g'(2^{w_{k+l}})}/{\g'(2^{w_k})}=2^l$.  First we establish   (\ref{1121}). Since $\g''>0$ on $\R^+$, we have
  $w_{k+l}>w_k$. From (\ref{p80}) we can see
$$2^{\A_0'({w_{k+l}}-{w_k})}\le \frac{\g'(2^{w_{k+l}})}{\g'(2^{w_k})}\le 2^{\A_1'({w_{k+l}}-{w_k})}.$$
This inequality, combined with  ${\g'(2^{{w_{k+l}}})}/{\g'(2^{w_k})}=2^l$, implies  ${w_{k+l}}-w_k\in [l/\A_1',l/\A_0']$, which establishes
  (\ref{1121}) with $\beta_j=1/\A_j'$, $j=0,1$. Similarly, by a parallel approach, we can deduce (\ref{1122}) from (\ref{p81}). The only distinction lies in ${w_{k+l}}<w_k$, and in this case, ${w_{k+l}}-w_k\in [-l/\A_0',-l/\A_1']$.
\end{proof}
For any $\bar{k}\in\Z$, we define a specific mapping ${\Phi}:\Z\to \Z$ as follows:
\begin{equation}\label{star1}
{\Phi}(\bar{k})=\sup\{k\in\Z:\ 2^{\bar k}\g(2^k)\le1\}.
\end{equation}
 (For example, ${\Phi}(\bar{k})=[-\bar k/\alpha]$ when $\Gamma(|t|)=|t|^\alpha$.) Since our main results are associated with general radial phase functions, the introduction of this mapping is necessary.
For convenience, in the subsequent proof of our main results, we frequently employ the following notation:
  \beq\label{brie}
  \bar{k}_*:={\Phi}(\bar{k})\quad{\rm whenever}\quad \bar{k}\in\Z.
  \eeq
  Clearly, we have $2^{\bar k}\g(2^{\bar{k}_*+1})>1$.
 Using this inequality and  (\ref{p8}) we then deduce by a simple computation that  for all $l\in [0,\infty)$ and  $k\in\R$,
   \beq\label{p9}
   \begin{aligned}
2^{-\A_1 l-\A_1}\le&\   2^{\bar{k}}\g(2^{\bar{k}_*-l})\le\  2^{-\A_0 l}\quad{\rm and}\\
2^{\A_0 l-\A_1} \le&\  2^{\bar{k}}\g(2^{\bar{k}_*+l})\le\  2^{\A_1 l}.
\end{aligned}
\eeq
In particular, $2^{-\A_1}\le  2^{\bar{k}}\g(2^{\bar{k}_*})\le 1$.

Next, we present some basic estimates related to maps that satisfy a specific sparse condition. In particular, the map $\Phi$ introduced above satisfies this sparse condition.
\begin{lemma}\label{l2.1}
Suppose that the map $\phi:\Z\to \Z$ satisfies
 the sparse condition: There exists a positive  integer $n_0$ such that $|\phi(j_1)-\phi(j_2)|>1$ for any pair $(j_1,j_2)\in\Z^2$ with $|j_1-j_2|\ge n_0$.  Under this assumption,  the following estimates hold:\\
(i)
Let   $\mathcal{S}_m:=\{j\in\Z:\ \phi(j)\le m\}$ with $m\in \Z$.
Then for any $\e>0$, we have
\beq\label{sum1}
\sum_{j\in\mathcal{S}_m} 2^{\e\phi(j)}
\les_\e n_0 2^{m\e};
\eeq
(ii) For every $p\in[1,\infty)$, we have a variant of the Littlewood-Paley inequality
\beq\label{ineq1}
\|\big(\sum_{j\in\Z}|P_{\phi(j)}f|^2\big)^{1/2}
\|_{{L}^p}\les {n_0} \|f\|_{\HH^p},
 \eeq
  where  $\HH^p:=\HH^p(\R^n)$ represents the Hardy space. 
\end{lemma}
 Clearly, the $\HH^p$ norm on the right-hand side of (\ref{ineq1})  can be replaced by the $L^p$ norm whenever $p\in (1,\infty)$. Moreover, 
 for any $q\in\Z$, 
 \eqref{ineq1} holds with $\phi(j)$ replaced by  $\phi(j)+q$. 
\begin{proof}
($i$) By a routine calculation, we can bound the left-hand side of (\ref{sum1}) by
\beq\label{2.7}
\sum_{l\le m}\sum_{l-1\le \phi(j)<l}2^{\phi(j)\e}
\le \sum_{l\le m}2^{l\e}\sum_{l-1\le \phi(j)<l}1,
\eeq
where $\e>0$ was used.
Because $\phi$ satisfies  the sparse condition,  the second sum on the right-hand side of (\ref{2.7}) is bounded  by the integer $n_0$. Then   (\ref{sum1}) follows.\\
($ii$)   We define $S_{d}:=\{j\in\Z:\ j\equiv d~ (\text{mod} ~n_0)\}$ for $d=0,\cdots,n_0-1$. By
Minkowski's inequality,
\beq\label{2.8} \|\big(\sum_{j\in\Z}|P_{\phi(j)}f|^2\big)^{1/2}
\|_{ L^p}\le  \sum_{d=0}^{n_0-1}\|\big(\sum_{j\in S_d}|P_{\phi(j)}f|^2\big)^{1/2}
\|_{ L^p}.
\eeq
Given that $\phi$ satisfies the sparse condition, we can infer that the $\ell^2$ sum on the right-hand side of (\ref{2.8}) is bounded by $\|P_jf\|_{\ell^2_j}$. Consequently, (\ref{ineq1}) follows from the application of the Littlewood-Paley theory on the Hardy space. 
\end{proof}
Note that the above $\alpha_0,\alpha_1,\alpha_0',\alpha_1',\beta_0,\beta_1$ and $n_0$ (in the application of Lemma \ref{l2.1})  depend only on $\{\mathfrak{C}_{1,j},\mathfrak{C}_{2,j}\}_{j=1}^2$.
\subsection{$V_r$ and related function spaces}\label{othef}
From the Plancherel-P\'{o}lya inequality \cite{PP36,PP37}, it can be observed that for all $r\in [1,\infty)$, $B_{r,1}^{1/r}\hookrightarrow V_r\hookrightarrow B_{r,\infty}^{1/r},$ where the notation $B_{p,q}^s$ represents the inhomogeneous Besov space (see  \cite{Gra14,BCD}).
By utilizing the first embedding and recognizing the convenience of working with Besov space, it is sufficient to manage the $B_{r,1}^{1/r}$ norm in the proofs of our results. However, as we will later see, some crucial modifications, derived from the definition of the variational seminorm, are needed before bounding this norm.
Furthermore, we obtain from  the fundamental theorem of calculus  that for all $r\in [1,\infty)$,
\beq\label{rou1}
V_r\big({\bf a}_u:u\in \mathcal{K}\big)\le \|\p_u ({\bf a}_u)\|_{L^1(u\in \mathcal{K})}\ {\rm whenever}\ \mathcal{K}\ {\rm is \ a\ interval}.
\eeq
In fact,
this inequality is effective in controlling terms with favorable bounds.
\subsection{An elementary decomposition of $H^{(u)}f$} \label{bd}
We conclude this section by presenting an elementary decomposition of $H^{(u)}f$.
By utilizing the partition of unity $\sum_{l\in\Z}\psi_l(t)=1$, we have
\beq\label{sig}
H^{(u)}f(x)=\sum_{l\in\Z}H^{(u)}_lf(x),\quad {\rm where}\ H^{(u)}_lf(x):=\int f(x-t)e^{iu\g(|t|)}\psi_l(t)K(t)dt.
\eeq
It should be noted that $K(t)$ satisfies condition (\ref{v1}). Consequently, for all $u\in \R$ and $l\in\Z$, we can establish the following bound for $H^{(u)}_lf$:
\beq\label{v2}
|H^{(u)}_lf(x)|\les \frac{1}{2^{nl}}\int_{|t|\le 2^{l+1}}|f(x-t)|dt\les Mf(x), \quad x\in\R^n,
\eeq
where $M$ represents the Hardy-Littlewood maximal operator. In the subsequent analysis, we will frequently employ (\ref{v2}).
\section{Reduction of  the variational inequality (\ref{var})}
\label{redut1}
Let  ${\bf H}f:=\{H^{(u)} f: u\in\R\}$, and  denote
$${\bf H}_jf:=\big\{H^{(u)} f: u\in [2^j,2^{j+1}]\big\},\quad j\in\Z.$$
 Then
   the left-hand side of (\ref{var}) can be controlled  (up to a uniform constant)
by  the sum of
 the dyadic long variation
$V_r^{\rm dyad}({\bf H}f)$ and  the short variation $V_r^{\rm sh}({\bf H}f)$, where these variations are defined as follows:
$$
\begin{aligned}
V_r^{\rm dyad}({\bf H}f):=&\ \sup_{N\in\N}~\sup_{k_1< \cdots< k_N
}
\|H^{(2^{k_{i+1}})}f-H^{(2^{k_{i}})}f\|_{\ell^r(1\le i\le N-1)},\\
V_r^{\rm sh}({\bf H}f):=&\ \|V_r({\bf H}_jf)\|_{\ell^r(j\in\Z)}.
\end{aligned}
$$
 Since this procedure is standard (see e.g., \cite{JSW08,BMSW18}), 
we omit the details.
\begin{prop}\label{long0}
 let $p\in(1,\infty)$ and $r\in(2,\infty)$. Suppose the function $\g\in \U$. Then
\beq\label{ka1}
\|V_r^{\rm dyad}({\bf H}f)\|_p\les \|f\|_p.
\eeq
\end{prop}

\begin{prop}\label{short}
Let $p\in(1,\infty)$, $r\in(2,\infty)$ and $r>p'/n$.
Assume the function $\g\in \U$. Then
\beq\label{ka2}
\big\|V_r^{\rm sh}({\bf H}f)\big\|_p\les \|f\|_p.
\eeq
\end{prop}
We first prove (\ref{var})  under the assumption that Propositions \ref{long0} and \ref{short}  hold.
\begin{proof}[Proof of (\ref{var})]
As $V_r({\bf H}f)$ can be dominated (up to a uniform constant) by the sum of $V_r^{\rm dyad}({\bf H}f)$ and $V_r^{\rm sh}({\bf H}f)$, we conclude from Minkowski's inequality that (\ref{var}) directly follows from the estimates (\ref{ka1}) and (\ref{ka2}).
\end{proof}
In the following context, we will proceed with the sequential proof of Proposition \ref{long0} and  Proposition \ref{short}.
\section{Long variation-norm estimate}
\label{long}
In this section, we
demonstrate the long variation-norm estimate  (\ref{ka1}) in Proposition \ref{long0}.
By Lemma 2.1 in \cite{JSW08} (see also  \cite{Bour89}),
it suffices to  establish the desired   jump inequality, that is, for every $p\in (1,\infty)$,
\beq\label{ka}
\big\|\la\sqrt{N_\la\big(H^{(2^k)}f: k\in\Z\big)}\big\|_p
\les \|f\|_p
\eeq
uniformly in $\la>0$.
 Actually, these types of  jump inequalities   were considered   as the endpoint cases of the corresponding variation-norm estimates in the studies of  many important operators; moreover, they may be of independent interest  (see e.g., \cite{MSZ20,MSZ202}). We observe through a straightforward computation  that for any $\{{\bf a}_k\}_{k\in\Z}\in L^p(\ell^2)$,
\beq\label{tri1}
\Big\|\la\sqrt{N_\la\big({\bf a}_k: k\in\Z\big)}\Big\|_p\les \|{\bf a}_k\|_{L^p(\ell^2_k)}.
\eeq
 But this elementary inequality  is evidently insufficient to prove the desired estimate (\ref{ka}). Moreover, the derivation of (\ref{tri1}) does not  require exploiting the gain brought about by the difference between any two functions in $\{{\bf a}_k\}_{k\in\Z}$, as seen in the definition of the $\lambda$-jump function. This kind of gain,   however, will play a pivotal role in proving (\ref{ka}) (as seen in (\ref{za1}) below). To accomplish this, we will make use of the observation: the jump inequality  (\ref{ka}) remains unaffected if each function ${\bf a}_k$ is adjusted by  $\mathcal{A}$, a quantity that may depend on the variable $x$ but not on $k$.
\subsection{Reduction of (\ref{ka})}
For any $k,l\in\Z$,
  to abbreviate the notation, we denote
$$\h_{k,l}:=H^{(2^k)}_l.$$
For any $x\in\R^n$ and $k\in\Z$,
applying (\ref{sig}) with $u=2^k$,  we obtain the decomposition
$$H^{(2^k)}f(x)=\sum_{l\ge 0 }\h_{k,k_*-l}f(x)
+\sum_{l\ge1}\h_{k,k_*+l}f(x)=:H_{k,no}f(x)+H_{k,o}f(x),$$
where $k_*$ is defined as (\ref{brie}) with the substitution of $\bar{k}$ with $k$.
We  first explain how this partition works.
For every $k\in\Z$,
 $H_{k,no}f$ behaves analogously to a  singular integral because  $e^{i2^k\g(|t|)}\sim 1$ on supp $\psi_{k_*-l}(t)$ whenever $l\ge 0$, while $H_{k,o}f$ can be seen as  a sum of a sequence of   oscillatory integrals $\h_{k,k_*+l}f$ with $l\ge 1$. As a consequence of this discrepancy, we shall employ distinct strategies to control $H_{k,no}f$ and  $H_{k,o}f$.
To demonstrate (\ref{ka}),
 it suffices  to prove Lemmas \ref{low1} and \ref{high1} below.
\begin{lemma}\label{low1}
Let  $\g\in \U$ and $p\in (1,\infty)$. Then  we have
\begin{equation}\label{wan1}
\Big\| \lambda \sqrt{N_\lambda\big(H_{k,\text{no}}f: k\in\Z\big)} \Big\|_p \lesssim \|f\|_p
\end{equation}
uniformly in $\la>0$.
\end{lemma}
\begin{lemma}\label{high1}
Let  $\g\in \U$, $l\ge 1$ and $p\in (1,\infty)$. Then there exists a  constant $\e>0$ such that
\beq\label{hf1}
\|\h_{k,k_*+l} f\|_{L^p(\ell^2_k)}\les 2^{-l\e}\|f\|_p.
\eeq
\end{lemma}
Assuming Lemmas \ref{low1} and \ref{high1}, we prove the
Proposition \ref{long0} by showing the  jump inequality
(\ref{ka}).
\begin{proof}[Proof of Proposition \ref{long0}]
Summing over $l\ge 1$ on both sides of  (\ref{hf1}) yields the inequality
$
\|H_{k,o}f\|_{L^p(\ell^2_k)}\les \|f\|_p,
$
which, combined with  (\ref{tri1}),
implies  that
\beq\label{ol1}
\big\|\la\sqrt{N_\la\big(H_{k,o}f: k\in\Z\big)}\big\|_p\les \|f\|_p
\eeq
uniformly in $\la>0$.
  Then we obtain the desired (\ref{ka}) by combining (\ref{ol1}) and  Lemma \ref{low1}. 
\end{proof}
Next, we prove Lemma \ref{low1}
and Lemma \ref{high1} in order.
\subsection{Proof of Lemma \ref{low1}}
For $k\in\Z$, $m\in\Z^+$ and $l\in \Z_{\ge 0}$, we denote
\beq\label{99}
\begin{aligned}
    \psi_{\circ,k}(t):=&\ \sum_{l'\ge 0} \psi_{k-l'}(t),\\
    \Psi_{k,m,l}(t):=&\ \big[{\g(|t|)}/{\g(2^{k_*-l})}\big]^m \psi_{k_*-l}(t),\\
    \Psi_m^{k,l}(t):=&\ \Psi_{k,m,l}(2^{k_*-l}t)K(t),
\end{aligned}
\eeq
where  $k_*$ is defined by (\ref{brie}) with $\bar{k}$ replaced by $k$.
From the conditions (\ref{curv}) and (\ref{smooth}) we can see that $\{\Psi_m^{k,l}\}$ are  Schwartz functions supported  in $\{t\in\R^n:2^{-1}\le |t|\le 2\}$. In particular, we have by  a routine calculation that there is a positive constant   $C_\circ$ independent of $k$ and $l$ such that
\beq\label{upb}
\mathcal{B}_{m,k,l}:=\|\Psi_m^{k,l}\|_{W^{1,\infty}}\le C_\circ^m,
\eeq
 where
 $W^{1,\infty}$ is a non-homogeneous Sobolev space defined by
 $\|f\|_{W^{1,\infty}}:=\|f\|_\infty+\|\na f||_\infty$.
Using (\ref{p8}), we have
$2^k\g(|t|)\les2^k\g(2^{k_*})\les  1$
 whenever $|t|\les 2^{k_*}$. Applying
 Taylor's expansion
$e^{i2^k\g(|t|)}=\sum_{m=0}^\infty \frac{(i2^k\g(|t|))^m}{m!}$
and (\ref{99}), we
can write $H_{k,no}f$   as
\beq\label{zo1}
\begin{aligned}
H_{k,no}f(x)&=\
 \sum_{l\ge 0}\sum_{m=0}^\infty\int f(x-t) \frac{(i2^k\g(|t|))^m}{m!} \psi_{k_*-l}(t)K(t)dt\\
&=: \ T_{0,k}f(x)+\sum_{l\ge 0}\sum_{m=1}^\infty \frac{(i2^k\g(2^{k_*-l}))^m}{m!} T_{m,k,l}f(x),
\end{aligned}
\eeq
where the operators $T_{0,k}$ and $T_{m,k,l}$ are respectively defined by
\beq\label{no2}
\begin{aligned}
T_{0,k}f(x):=&\  {\rm p.v.}\int f(x-t) \psi_{\circ,k_*}(t)K(t)dt,\\
T_{m,k,l}f(x):=&\ \int f(x-t)  \Psi_{k,m,l}(t)K(t)dt.
\end{aligned}
\eeq
Then we can reduce the matter to proving the associated jump inequalities for each term appearing on the right-hand side  of
(\ref{zo1}).  By (\ref{p8}) and a routine computation, we  deduce
\beq\label{q1}
\mathcal{B}_{m,k,l}^{-1}|T_{m,k,l}f(x)|\les M f(x).
\eeq

We first bound $T_{0,k}f$.
Establishing the desired jump inequality directly for the singular integral $T_{0,k}f(x)$ appears challenging. To circumvent the singularity, as mentioned earlier below (\ref{tri1}), we  exploit this observation: the intended estimate (\ref{wan1}) remains unaffected if we substitute $T_{0,k}f(x)$ with
\beq\label{za1}
\mathcal{T}_{0,k}f(x):=-\sum_{n\ge k_*}f*\sigma_{n}(x),
\eeq
where  $\mathcal{F}\{\sigma_{n}\}(\cdot):=\hat{\sigma}(2^{n}\cdot)$ and
$ \langle\sigma,f\rangle:=\int f(t)\psi (t)K(t)dt.$
Note that
 $\hat{\sigma}(0)=0$  and $\sigma$ is a compactly supported finite Borel measure.
 Next, we  bound the jump inequality for
$\mathcal{T}_{0,k}f$.
It follows  by Theorem 1.2 in \cite{JSW08} that for  $p\in(1,\infty)$,
$$\big\|\la\sqrt{N_\la\big(\mathcal{T}_{0,k}f: k\in\Z\big)}\big\|_p\les
\big\|\la\sqrt{N_\la\big(\sum_{n\ge k}f*\sigma_{n}: k\in\Z\big)}\big\|_p\les
 \|f\|_p$$
 uniformly in $\la>0$. As a result,   for  $p\in(1,\infty)$, we obtain
$$
\sup_{\la>0}\big\|\la\sqrt{N_\la\big({T}_{0,k}f: k\in\Z\big)}\big\|_p\les
 \|f\|_p.
$$

To finish the proof of  (\ref{low1}),
it suffices  to  bound the desired jump inequality for the second term on the right-hand side of
(\ref{zo1}).
 We claim that for each $p\in (1,\infty)$,
\begin{eqnarray}
\|T_{m,k,l}f\|_{L^p(\ell^2_k)}\les \mathcal{B}_{m,k,l} \|f\|_p.\label{as2}
\end{eqnarray}
Using Minkowski's inequality,
 (\ref{tri1}),
(\ref{p9}) with $\bar{k}=k$,
and the above claim (\ref{as2}) in order, we can deduce that the
 $L^p(\ell^2_k)$ norm of the second term on the right-hand side of
(\ref{zo1}) is
$
\les \sum_{l\ge 0}\sum_{m=1}^\infty  \frac{\mathcal{B}_{m,k,l}}{m!} \frac{1}{2^{\A_0 ml}} \|f\|_p,$ which, together with the uniform estimate  (\ref{upb}),
  implies the desired  jump inequality.
We next aim at proving the above claim (\ref{as2}). By the Littlewood-Paley  decomposition $f=\sum_{j\in\Z}P_{j
+l-k_*}f$, 
it is sufficient  to prove
that there exists an   $\e>0$   such that  for  $p\in(1,\infty)$
\beq\label{as02}
\mathcal{B}_{m,k,l}^{-1} \|T_{m,k,l}P_{j+l-k_*}f\|_{L^p(\ell^2_k)}\les 2^{-\e|j|} \|f\|_p.
\eeq
Utilizing   (\ref{q1}), the
Fefferman-Stein inequality and  Lemma \ref{l2.1} in sequence, we can easily obtain that  for  $p\in(1,\infty)$,
$$
\begin{aligned}
&\
\mathcal{B}_{m,k,l}^{-1}\|T_{m,k,l}P_{j+l-k_*}f\|_{L^p(\ell^2_k)}
\les\  \|M P_{j+l-k_*}f\|_{L^p(\ell^2_k)}
\les\  \|P_{j+l-k_*}f\|_{L^p(\ell^2_k)}\les \|f\|_p.
\end{aligned}
$$
Thus,
by interpolation,  it suffices to prove
\beq\label{98}
\mathcal{B}_{m,k,l}^{-1}\|\mathcal{F}^{-1}\{ \widehat{f}\ \mathcal{M}_{j,l,k_*}\}\|_{L^2(\ell^2_k)}\les 2^{-\e|j|} \|f\|_2,
\eeq
 where the multiplier $\mathcal{M}_{j,l,k_*}$ is given by
\beq\label{as91}
\begin{aligned}
\mathcal{M}_{j,l,k_*}(\xi)=&\ \psi(\frac{\xi}{2^{j+l-k_*}}) \int  e^{-it\cdot \xi}  \Psi_{k,m,l}(t)K(t)dt\\
=&\ \psi(\frac{\xi}{2^{j+l-k_*}}) \int  e^{-i2^{k_*-l}t\cdot \xi} \Psi_m^{k,l}(t)dt.
\end{aligned}
\eeq
By Plancherel's identity, it is sufficient for (\ref{98}) to prove
the uniform inequality
\beq\label{97}
|\mathcal{M}_{j,l,k_*}(\xi)|\les \mathcal{B}_{m,k,l} \  \psi(\frac{\xi}{2^{j+l-k_*}}) 2^{-|j|}.
\eeq
Integrating  by parts in $t$, we  obtain from (\ref{as91}) that
\beq\label{ren1}
|\mathcal{M}_{j,l,k_*}(\xi)|\les \psi(\frac{\xi}{2^{j+l-k_*}})
|\frac{\xi}{2^{l-k_*}}|^{-1}\|\na \Psi_m^{k,l}\|_{L^\infty}.
\eeq
On the other hand,
 using  $\int \Psi_m^{k,l}(t) dt=0$ derived from $\int_{\mathbb{S}^{n-1}}\Omega (\theta)d\sigma=0$, we can infer
\beq\label{ren2}
|\mathcal{M}_{j,l,k_*}(\xi)|
\les   \psi(\frac{\xi}{2^{j+l-k_*}})
|\frac{\xi}{2^{l-k_*}}|\|\Psi_m^{k,l}\|_{L^\infty}.
\eeq
Since  $|\xi|\sim 2^{j+l-k_*}$ and (\ref{upb}),
we finally obtain
 (\ref{97})  by combining (\ref{ren1}) and  (\ref{ren2}). This completes the proof of  Lemma \ref{low1}.
\subsection{Proof of Lemma \ref{high1}}
 Applying the Fourier inverse transform, we can  write the  multiplier of the operator $\h_{k,k_*+l}$ as
\beq\label{007}
m_{k,l}(\xi)
:=\int e^{iA_{k,l}(t,\xi)}\tilde{\psi}(t)dt,
\eeq
where the functions $A_{k,l}$ and $\tilde{\psi}$ are given by
\beq\label{smoo}
A_{k,l}(t,\xi):=2^k\g(2^{l+k_*}|t|)-2^{l+k_*}\xi\cdot t\ \ {\rm and} \ \ \tilde{\psi}(t)= \psi(t)K(t).
\eeq
An elementary
 computation gives
$$\na_t(A_{k,l}(t,\xi))
=  t|t|^{-1}2^{k+l+k_*}\g'(2^{l+k_*}|t|)-2^{l+k_*}\xi,$$
which may imply that  there exists a positive integer  $\nu$ (large enough) such that
\beq\label{e1}
|\na_t(A_{k,l}(t,\xi))|\ge 2^{
l/{\nu}}
\eeq
whenever $|\xi|\gtrsim 2^{{\nu} l-k_*}$ or $|\xi|\les 2^{-{\nu} l-k_*}$. Indeed,
for $l\ge1$,
using $|t|\sim 1$,  (\ref{p8}), (\ref{curv}) and (\ref{smooth}) with $j=1$, along with (\ref{p9}), we derive
\beq\label{m2}
2^{\A_0l}\les
2^{k+l+k_*}\g'(2^{l+k_*}|t|)
= \frac{2^{l+k_*}\g'(2^{l+k_*}|t|)}{\g(2^{l+k_*})} \frac{\g(2^{l+k_*})}{\g(2^{k_*})}
2^k\g(2^{k_*})
\les 2^{\A_1l},
\eeq
and subsequently establish (\ref{e1})
by choosing $\nu$ sufficiently large.
  Therefore, relying on (\ref{e1}), we may anticipate that the integration by parts will yield the desired estimate whenever $|\xi| \gtrsim 2^{{\nu} l-k_*}$ or $|\xi| \lesssim 2^{-{\nu} l-k_*}$. However, for the remaining values of $\xi$, where the phase $A_{k,l}(t,\xi)$ may exhibit a critical point, we must adopt a different approach. As we will see later, the corresponding estimates for the remaining $\xi$ will hinge on the crucial property that the hyper-surface $(t,2^k\g(2^{l+k_*}|t|))$ possesses non-vanishing Gaussian curvature.
Motivated by  the above discussions, we utilize  the Littlewood-Paley decomposition to get
$$\h_{k,k_*+l} f(x)=\sum_{j\ge {\nu}l}\h_{k,l}^jf(x)
+\sum_{j\le  -{\nu}l}\h_{k,l}^jf(x)
+\sum_{|j|<{\nu}l}\h_{k,l}^jf(x),$$
where $\nu$ is given by (\ref{e1}), and
$$\h_{k,l}^j:=\h_{k,k_*+l}P_{j-k_*}.$$
Let $m_{k,l}^j(\xi)$ denote the  associated
 multiplier of $\h_{k,l}^j$  described  by
\beq\label{cheng1} m_{k,l}^j(\xi):=\psi(\frac{\xi}{2^{j-k_*}})m_{k,l}(\xi)
\eeq
where $m_{k,l}$ is defined in (\ref{007}).
Hence,
it is sufficient for  (\ref{hf1}) to show that there is an $\e>0$ such that
for $p\in(1,\infty)$,
\begin{align}
\|\sum_{|j|< {\nu}l}\h_{k,l}^j f\|_{L^p(\ell^2_k)}\les&\   2^{-l\e}\|f\|_p,\label{l12}\\
\|\sum_{j\le - {\nu}l}\h_{k,l}^j f\|_{L^p(\ell^2_k)}\les&\   2^{-l\e}\|f\|_p\quad \label{l11}{\rm and}\\
\|\sum_{j\ge {\nu}l}\h_{k,l}^j f\|_{L^p(\ell^2_k)}\les&\  2^{-l\e}\|f\|_p \label{l10}.
\end{align}
From  (\ref{ineq1}) we  can infer that  for every $p\in(1,\infty)$,
\beq\label{basic}
\|\h_{k,l}^j f\|_{L^p(\ell^2_k)}\les \|M P_{j-k_*}f\|_{L^p(\ell^2_k)}\les \|f\|_p
\eeq
with the implicit  constant  independent of
 $(k,j,l)$.

We first treat (\ref{l12}). By (\ref{basic}) and  interpolation, it suffices to show the case $p=2$. By Plancherel's identity,
we further reduce the matter to proving the uniform estimate
\beq\label{goo1}
\|m_{k,l}^j(\xi)\|_{\ell^2_k}\les 2^{-l\e}
\eeq
with the implicit  constant  independent of
 $j$, $l$ and $\xi$.
Applying (\ref{smoo}) and
 the polar coordinates, we can rewrite $m_{k,l}(\xi)$ as
$$m_{k,l}(\xi)=\int_{\mathbb{S}^{n-1}}\Omega(\theta)\big(
\int r^{-1}{\psi}(r) e^{iA_{k,l}(r\theta,\xi)}dr
\big)d\sigma.$$
If the lower bound estimate
\beq\label{a61}
|\frac{\p^2}{\p r^2}(A_{k,l}(r\theta,\xi))|=2^k2^{2(l+k_*)}|\g''(2^{l+k_*}r)|\gtrsim 2^{\A_0l}
\eeq
holds, where  $\A_0$ is given by   (\ref{p8}),
we can deduce,  by van der Corput's lemma, that   $|m_{k,l}(\xi)|\les 2^{-\A_0 l/2}$. Then
the desired (\ref{goo1}) with $\e=\A_0/2$ follows from (\ref{cheng1}) since $\sum_{k\in\Z}\psi_{j-k_*}(\xi)\les 1$.
So it remains to prove (\ref{a61}).
 Using the condition (\ref{curv}) with $j=2$,  (\ref{p8}) and (\ref{p9}), we have
$$|\frac{\p^2}{\p r^2}(A_{k,l}(r\theta,\xi))|
=\frac{2^{2(l+k_*)}|\g''(2^{l+k_*}r)|}{\g(2^{l+k_*}r)}\frac{\g(2^{l+k_*}r)}{\g(2^{k_*}r)} \frac{\g(2^{k_*}r)}{\g(2^{k_*})}  2^k\g(2^{k_*})
\gtrsim 2^{\A_0l},\quad\text{as desired.}$$

We next prove  $(\ref{l11})$.
For  $j\le -\nu  l$ and $|\xi|\sim 2^{j-k_*}$,  we have Taylor's expansion    $e^{-i2^{l+k_*}\xi\cdot t}=\sum_{m=0}^\infty \frac{(-i2^{l+k_*}\xi\cdot t)^m}{m!}$. This with
$\int e^{iA_{k,l}(t,0)}\tilde{\psi}(t)dt=0$ implies
\beq\label{goo2}
\h_{k,l}^jf(x)=\sum_{m=1}^\infty\frac{(-i)^m}{m!} {2^{(j+l)m}}\  T_m P_{j-k_*}f(x),
\eeq
where the operators $\{T_m\}$ are given by
\beq\label{comp}
T_m f(x)=\int \widehat{f}(\xi) e^{i\xi\cdot x}
\left(\int (\frac{\xi}{2^{j-k_*}} \cdot t)^m e^{iA_{k,l}(t,0)
}\tilde{\psi}(t)dt\right)d\xi.
\eeq
By expanding  the term $(2^{k_*-j}\xi\cdot t)^m$ in (\ref{comp}), 
 $|T_m P_{j-k_*}f|$ can be bounded by
the sum of $O(n^m)$ terms of the form
$2^m |\tilde{P}_{j-k_*,m}f| $,
where the operators
$\{\tilde{P}_{j-k_*,m}\}$  are  variants of the Littlewood-Paley projection $P_{j-k_*}$ (see Lemma \ref{l2.1}). In particular,  for every $m\ge1$ and  $p\in(1,\infty)$,
$
\|\tilde{P}_{j-k_*,m}f\|_{L^p(\ell^2_k)}
\les 4^m\|f\|_p. $
Combining this estimate with (\ref{comp}) implies that
 there exists a positive constant  $C_*$ independent of $m$ such that for every $m\ge1$ and  $p\in(1,\infty)$,
\beq\label{go9}
\|T_m P_{j-k_*}f\|_{L^p(\ell^2_k)}
\les C_*^m\|f\|_p,
 \eeq
 where the implicit constant is independent of $m$. Then it follows from the combination of
 (\ref{go9}) and (\ref{goo2}) that  the $L^p$ norm of $\h_{k,l}^jf$ can be bounded  by
 $\sum_{m=1}^\infty\frac{C_*^m}{m!} {2^{(j+l)m}}$, up to a uniform  constant. Then we achieve (\ref{l11}) by
 $j\le -\nu l$ and a routine calculation.

Finally, we prove (\ref{l10}). Since $j\ge {\nu}l$ and $\nu$ is sufficiently large,  we have (\ref{e1}).
As in the previous statements,
we will use the integration by parts (in the variable $t$) to the multiplier $m_{k,l}^j$.  We define  the differential  operator $\mathcal{L}$  by
$
\mathcal{L}(g):=i\frac{2^{l+k_*}(\xi\cdot \na_t) g}{|2^{l+k_*}\xi|^2},
$
and denote by $^t\mathcal{L}$ the transpose of $\mathcal{L}$.  Integrating  by parts in $t$ gives
\beq\label{inte01}
\begin{aligned}
m_{k,l}^j(\xi)=&\ \psi(\frac{\xi}{2^{j-k_*}})\int \mathcal{L}(e^{-i2^{l+k_*}\xi\cdot t}) e^{iA_{k,l}(t,0)}\tilde{\psi}(t)dt\\
=&\ \psi(\frac{\xi}{2^{j-k_*}})\int e^{-i2^{l+k_*}\xi\cdot t}\  (^t\mathcal{L})\big( e^{iA_{k,l}(t,0)}\tilde{\psi}(t)\big)dt,
\end{aligned}
\eeq
where
$\mathcal{L}(e^{-i2^{l+k_*}\xi\cdot t})=e^{-i2^{l+k_*}\xi\cdot t}$ was used in the first equality of (\ref{inte01}).
Expanding this term $(^t\mathcal{L})(\cdot)$ in the second integral of (\ref{inte01}), and applying the upper bound in (\ref{m2}),
 we can write $\h_{k,l}^j$ as
the sum of two terms like
$2^{-(j+l-\A_1 l)}\tilde{\h}_{k,l}^jf(x),$
up to a uniform constant,
where $\tilde{\h}_{k,l}^j$ is a variant of ${\h}_{k,l}^j$, and it satisfies an estimate similar to  (\ref{basic}).
Thus, to see (\ref{l10}), it  suffices to show
that for  every $l\ge 1$ and every $j\ge \nu l$,
\beq\label{m3}
2^{-(j+l-\A_1 l)} \|\tilde{\h}_{k,l}^jf\|_{L^p(\ell^2_k)}
\les 2^{-j/2}\|f\|_p.
\eeq
Invoking  that $\nu$ is  sufficiently  large, we immediately achieve (\ref{m3}) from Lemma \ref{l2.1}.
This completes the proof of (\ref{l10}), and we finish the  proof of  Lemma \ref{high1}.
\section{A reduction of short variational estimate}
\label{s5p}
\subsection{An equivalent definition of $V_r({\bf H}_j f)$}\label{equ21}
Making the  change of variable $u\to 2^ju$,
we can write  $V_r({\bf H}_jf)$
as
\beq\label{vr}
V_r({\bf H}_jf):=V_r\big( H^{(2^ju)}f: u\in [1,2]\big).
\eeq
 As the statements in Subsection \ref{othef},   we  choose to   bound
 the   $B_{r,1}^{1/r}$ norm of $\chi(u) H^{(2^ju)}f$,
  where $\chi:\R\to [0,1]$ denotes   a  smooth function with compact support in $[2^{-1},2^2]$ and equals one on $[1,2]$.
  However,  it seems difficult to bound the $B_{r,1}^{1/r}$ norm of $\chi(u) H^{(2^ju)}f$  directly  since   the operator $H^{(2^ju)}$ behaves similarly to
a   singular integral operator  whenever $e^{i2^ju\g(|t|)}\sim 1$.
 More precisely,  if we  use the decomposition (recall \eqref{sig})
 \beq\label{m1}
 H^{(2^ju)}f(x)=\sum_{l\le 0}H^{(2^ju)}_{l+j_*}f(x)
+\sum_{l> 0}H^{(2^ju)}_{l+j_*}f(x),
\eeq
 it will be difficult to bound the associated estimate for the first term on the right-hand side of (\ref{m1}).
 To overcome this difficulty, motivated by the arguments below (\ref{tri1}), for each $u$, we
 substitute $H^{(2^ju)}f(x)$ by $H^{(2^ju)}f(x)+\mathcal{A}_1$, where  $\mathcal{A}_1=-\sum_{l\le 0}H^{(0)}_{l+j_*}f(x)$. This substitution yields
a new operator  $T_{j,l}[u]$ which can be described  by
\beq\label{dft1}
T_{j,l}[u]f(x):=\int f(x-t)(e^{i2^ju\g(|t|)}-1) \frac{1}{2^{n(l+j_*)}}\tilde{\psi}(\frac{t}{2^{l+j_*}})dt,
\eeq
where $\tilde{\psi}$ is given by (\ref{smoo}).
Consequently,
we can  replace $H^{(2^ju)}f$ in (\ref{vr}) by 
\beq\label{tc}
H_\circ^{(2^ju)}f(x)=:
\sum_{l\le 0}
T_{j,l}[u]f(x)+ \sum_{l> 0}H^{(2^j u)}_{l+j_*}f(x).
\eeq
Now, we can    translate the estimate for $\sum_{l\le 0}H^{(2^j u)}_{l+j_*}f$ into an estimate for  $\sum_{l\le0}T_{j,l}[ u]f$.  Moreover, following   the arguments handling  the second term on the right-hand side of (\ref{zo1}), we  can obtain
the desired estimate for $T_{j,l}[ u]$ (see Lemma \ref{l1} below).
Finally, we  rewrite  $V_r({\bf H}_jf)$ in a new form
\beq\label{edf}
V_r({\bf H}_jf)=V_r\big(H^{(2^ju)}_\circ f:u\in [1,2]\big).
\eeq
\subsection{Reduction of Proposition \ref{short}}
\label{r5.2}
Proposition \ref{short} is a direct result of   Lemmas \ref{l1} and \ref{l2} below.
\begin{lemma}\label{l1}
For $l\le 0$, $r\in [2,\infty]$ and $p\in(1,\infty)$.
Suppose  the function $\g\in \U$.
Then there exists a positive constant  $\e$ such that
\beq\label{aii1}
\Big\|\|\chi(u) T_{j,l}[u]f\|_{B_{r,1}^{1/r}(u\in\R)}\Big\|_{L^p(\ell^r_j)}\les 2^{ \e l} \|f\|_p.
\eeq
\end{lemma}
\begin{lemma}\label{l2}
For $l\ge 1$, $p\in(1,\infty)$, $r\in(2,\infty)$ and $r>p'/n$.
Assume the function $\g\in \U$.
Then there exists a positive constant  $\e$ such that
$$
\Big\|\|\chi(u) H^{(2^j u)}_{l+j_*}f\|_{B_{r,1}^{1/r}(u\in\R)}\Big\|_{L^p(\ell^r_j)}\les 2^{ -\e l} \|f\|_p.
$$
\end{lemma}

We proceed to establish Proposition \ref{short}, assuming the validity of Lemmas \ref{l1} and \ref{l2}.
\begin{proof}[Proof of Proposition \ref{short}]
By (\ref{edf}),  it  now suffices to show the desired short variation-norm estimate for $H^{(2^ju)}_\circ$.  Utilizing (\ref{tc}), $B_{r,1}^{1/r}\hookrightarrow V_r$ and Minkowski's inequality,  we can achieve this estimate  by directly using  Lemmas \ref{l1} and \ref{l2}.
\end{proof}

It remains to show Lemmas \ref{l1} and \ref{l2}.  It is worth noting that $T_{j,l}[u]f$ is related to an essentially non-oscillatory integral. To elaborate further, expressing $T_{j,l}[u]f$ as the sum of convolutions with kernels meeting mean value zero and desired decay estimates allows us to establish its corresponding square-function estimate using Lemma \ref{l2.1}.
However, proving Lemma \ref{l2}, a key innovation in this paper, is more intricate. This complexity arises because we are currently dealing with an oscillatory integral, and its estimate necessitates the application of various techniques, including the Stein-Tomas restriction theorem. Therefore, in this section, we present the proof of the former, while the proof of the latter is deferred to the next section.
\vskip.1in
\begin{proof}[Proof of Lemma \ref{l1}]
It suffices to show the case  $r=2$ because $B_{2,1}^{1/2}\hookrightarrow B_{r,1}^{1/r}$ and
$\ell^2\hookrightarrow \ell^r$ for any $r>2$.
Using   (\ref{p9}), $1\le u\le 2$ and  $|t|\sim 2^{l+j_*}\les 2^{j_*}$, we have $2^ju\g(|t|)\les 1$. Then we have
 Taylor's expansion $e^{i2^ju\g(|t|)}=\sum_{m=0}^\infty\frac{(i2^ju\g(|t|))^m}{m!} $, which gives
$$
\begin{aligned}
T_{j,l}[u]f(x)=&\ \sum_{m=1}^\infty \int f(x-t)\frac{(i2^ju\g(|t|))^m}{m!}  \frac{1}{2^{n(l+j_*)}}\tilde{\psi}(\frac{t}{2^{l+j_*}})dt\\
=&\ \sum_{m=1}^\infty \frac{i^m}{m!}\big(2^ju\g(2^{l+j_*})\big)^m  (f*\rho_{m,l+j_*})(x),
\end{aligned}
$$
where  $\rho_{m,l+j_*}$ is a smooth function defined by
$$\rho_{m,l+j_*}(t)=\frac{1}{2^{n(l+j_*)}}\rho_m(\frac{t}{2^{l+j_*}}),\ \  \rho_m(t)=\left(\frac{\g(2^{l+j_*}|t|)}{\g(2^{l+j_*})}\right)^m  \tilde{\psi}(t).$$
By (\ref{smoo}),  $\int_{\mathbb{S}^{n-1}}\Omega(x)d\sigma(x)=0$ and (\ref{p9}),   we deduce 
\beq\label{rh}
\int \rho_m(t) dt=0\quad{\rm and}\quad |\rho_m(t)|+|\na\rho_m(t)|\les C^m(1+|t|)^{-n-1}
\eeq
for some uniform $C>0$.  In what follows, the uniform constant $C$ may vary at each appearance.
Using (\ref{rh}) and Lemma \ref{l2.1}, we may further deduce
\beq\label{squ1}
\|f*\rho_{m,l+j_*}\|_{L^p(\ell^2_j)}\les C^m \|f\|_p.
\eeq
In addition,  by
$\||\cdot|^m\chi(\cdot)\|_{B_{2,1}^{1/2}}\les C^m$ and (\ref{p9}),
the left-hand side of (\ref{aii1}) is majorized   (up to an absolute constant) by
\beq\label{squ2}
 2^{\A_0 l}\sum_{m\ge 1}\frac{C^m}{m!}
\|f*\rho_{m,l+j_*}\|_{L^p(\ell^2_j)}
\eeq
with $\A_0$  as in (\ref{p8}). In the end,  we achieve  (\ref{aii1}) with $\e=\A_0$   by  inserting  (\ref{squ1}) into (\ref{squ2}).
\end{proof}
To finish the proof of  Proposition \ref{short}, it remains to prove Lemma \ref{l2}.
\section{Short variation-norm estimate: proof of Lemma \ref{l2}}
\label{Pftheorem1}
In this section, the proof of  Lemma \ref{l2} is reduced to  establishing Propositions \ref{short2} and \ref{short1} below. These propositions will be demonstrated by leveraging the combination of Lemma \ref{4l} with Lemma \ref{400l} and Lemma \ref{51}, respectively. The proofs of these lemmas employ diverse techniques, including the method of stationary phase (see Subsection \ref{sdef1}), Stein-Tomas restriction estimates, a square-function estimate (see Lemma \ref{6l}), and a localization technique (see Lemma \ref{local}).
\subsection{Useful notations}\label{fenno}
Let $j\in\Z$, $l\ge 1$, $r\in [2,\infty)$,  $p\in[1,\infty)$,   let $S$ be a subset of  $\Z$, and denote 
\beq\label{nn}
\begin{aligned}
\h_{j,l}f(u,x):=&\ \chi(u) H^{(2^j u)}_{l+j_*}f(x),\quad(u,x)\in \R\times\R^n,\\
\|\h_{j,l}f\|_{W_{r,p}(j\in S)}:=&\ \big\|\|\h_{j,l}f(u,x)\|_{B_{r,1}^{1/r}(u\in\R)}\big\|_{\mathbb L^p(x\in\R^n;\ell^r_j(S))},\\
\|\h_{j,l}f\|_{w_{r,p}(j\in S)}:=&\ \big\|\|\h_{j,l}f(u,x)\|_{L^r(u\in\R)}\big\|_{\mathbb L^p(x\in\R^n;\ell^r_j(S))},
\end{aligned}
\eeq
where $\mathbb{L}^p:=L^p$ for $p\in (1,\infty)$, and $\mathbb{L}^1:=L^{1,\infty}$.
  Note that we may apply the interpolation inequality
 \beq\label{cha1}
 \|\h_{j,l}f(\cdot,x)\|_{B_{r,1}^{1/r}}
\les \|\h_{j,l}f(\cdot,x)\|_r+\|\h_{j,l}f(\cdot,x)\|_r^{1-1/r}\|\h_{j,l}f(\cdot,x)\|_{\dot{W}^{1,r}}^{1/r}
\eeq
to establish a connection between the ${W_{r,p}(j\in S)}$ norm and the ${w_{r,p}(j\in S)}$ norm, where $\dot{W}^{1,r}$ represents the homogeneous Sobolev space defined by $\|f\|_{\dot{W}^{1,r}}=\|\nabla f\|_r$. Additionally, we will make use of interpolations involving the  spaces ${L}^{1,\infty}(\R^n)$ and ${\HH}^1(\R^n)$, addressing the endpoint case $p=1$.
\subsection{Reduction of Lemma \ref{l2}}\label{rel2}

The proof of Lemma \ref{l2} will be reduced  to establishing Propositions \ref{short2} and \ref{short1} below. As we will see later, each proof is founded on the following decomposition:
\beq\label{lo1}
\h_{j,l}f=\sum_{k\in\Z}\h_{j,l}P_{k-j_*}f,
\eeq
 and    the studies  of the associated multiplier
 \beq\label{mui0}
m_{k,j,l}(\xi):=\psi(\frac{\xi}{2^{k-j_*}})\int e^{iA_{j,l}^u(t,\xi)}\tilde{\psi}(t)dt
\eeq
with the phase $A_{j,l}^u(t,\xi)$  given by
\beq\label{Adef}
A_{j,l}^u (t,\xi):=u2^j\g(2^{l+j_*}|t|)-2^{l+j_*}t\cdot\xi.
\eeq
Keep in mind that  the hyper-surface $(t,2^j\g(2^{l+j_*}|t|))$ has non-vanishing Gaussian curvature whenever $\g\in \U$ and $|t|\sim 1$.
Note that we can rewrite $\h_{j,l}P_{k-j_*}f$  as a convolution, that is,
\beq\label{k11}
\h_{j,l}P_{k-j_*}f(u,x)=\chi(u)(f*_xK_{k,j,l})(u,x),
\eeq
where the kernel function $K_{k,j,l}$ is given by
$$K_{k,j,l}(u,x):=2^{n(k-j_*)}
\iint\psi(\xi)
e^{i{2^{k-j_*}}\xi\cdot(x-2^{l+j_*}t)+i2^ju\g(2^{l+j_*}|t|)}\tilde{\psi}(t)dtd\xi.
$$
By
integrating by parts in $\xi$, we  can infer that    for any $N\in\N$,
\beq\label{do1}|K_{k,j,l}(u,x)|\les_{N} \frac{2^{n(k-j_*)}}{(1+|2^{k-j_*}x|^{2N})}
\les_N 2^{-N(k+l)} \E_{k-j_*,N}(x)
\eeq
whenever $  |x|\ge 2^{l+j_*+10}$. See  \eqref{bei} below  for the definition of  $\E_{k-j_*,N}$.  

In the following context, contrary to the earlier discussions, we cannot straightforwardly apply the upper bound in (\ref{p8}) to quantify the term
$$\g_{lj_*}(1):={\g(2^{l+j_*})}/{\g(2^{j_*})}$$
or analogous terms. Otherwise, there is an inevitable risk of losing some range of values for $p$ when employing interpolations.
To bridge this gap, our approach involves partitioning the index set $\{j : j \in \Z\}$ in the sum into a union of subsets $\{S_{l, m}\}_{{m \in \mathcal{P}_l}}$, and applying the more accurate estimate $\g_{lj_*}(1) \sim 2^{m+l}$ within each subset $S_{l, m}$,
where the sets
 $ S_{l,m}, \mathcal{P}_l$ are defined respectively   by
\begin{align}
 S_{l,m}:=&\ \{j\in\Z:\ 2^{m+l}\le \g_{lj_*}(1)<2^{m+l+1}\},\label{slim}\\
\mathcal{P}_l:=&\ \{m\in\Z:\ \A_0l \le m+l\le \A_1 l \}.\label{Go5}
\end{align}
In the above arguments we  confined the values of $m$ to a smaller set $\mathcal{P}_l$ instead of considering the entire set $\Z$, based on (\ref{p8}). Consequently, we can reduce
the proof of Lemma \ref{l2} to establishing the validity of the following propositions: Propositions \ref{short2} and \ref{short1}. These propositions address scenarios involving small $p$ (near the endpoint) and large $p$, respectively.
\begin{prop}\label{short2}
Let $n\ge 2$, $l\ge 1$,  $m\in\mathcal{P}_l$ with $\mathcal{P}_l$ defined by (\ref{Go5}), and let $S_{l,m}$ be the set defined by (\ref{slim}). Suppose $1<p<2n/(2n-1)$ and  $p'/n<r\le p'$. Then
\beq\label{Go2}
\|\h_{j,l}f\|_{W_{r,p}(j\in S_{l,m})}\les 2^{(m+l)(1/r-n/p') } \|f\|_p,
\eeq
where the implicit constant  is independent of $m,l$.
\end{prop}
\begin{prop}\label{short1}
Let $n$, $l$, $m$ and $S_{l,m}$ be as in Proposition \ref{short2}. Suppose  $ 2\le r<\infty$ and ${2(n+1)}/{(n-1)}< p<\infty$. Then there exists a positive constant  $\e$ such that
\beq\label{ca1}
\|\h_{j,l}f\|_{W_{r,p}(j\in S_{l,m})} \les 2^{- (m+l)\e}\|f\|_p,
\eeq
where the implicit constant  is independent of $m,l$.
\end{prop}
With the above propositions in hand, we proceed to prove  Lemma \ref{l2}.
\begin{proof}[Proof of Lemma \ref{l2}]
By interpolating (\ref{Go2}) from Proposition \ref{short2} with (\ref{ca1}) from Proposition \ref{short1}, we can
show that for $l\ge 1$,  $p\in(1,\infty)$, $r\in(2,\infty)$ and $r>p'/n$,
\beq\label{indi}
\|\h_{j,l}f\|_{W_{r,p}(j\in S_{l,m})}
\les 2^{-(m+l)\e }\|f\|_p
\eeq
for some $\e>0$. In fact, one can see in Figure \ref{fig:1} that the ranges of $(p,r)$
in  Propositions \ref{short1} and \ref{short2}  are the interior of the rectangle AODB  and the interior of the triangle CEF, respectively.
Keeping  in mind  that
$\#\mathcal{P}_l\les l$, we can finish the proof of Lemma \ref{l2}
 by combining the simple inequality
$\|\h_{j,l}f\|_{W_{r,p}(j\in \Z)}\le \sum_{m\in\mathcal{P}_l}\|\h_{j,l}f\|_{W_{r,p}(j\in S_{l,m})}$
and (\ref{indi}).
\end{proof}

Note that  we may deduce  the estimate 
\beq\label{nod}
\|\h_{j,l}P_{k-j_*}f\|_{w_{r,p}(j\in S_{l,m})}\les \|f\|_{{\HH}^p}
\eeq
whenever $n\ge 2$,  $r\in [2,\infty]$, $p\in [1,\infty)$ and $k,l\in\Z$.
Indeed, by using the straightforward inequality $|\h_{j,l}f|\les Mf$ and  the 
$L^1\to L^{1,\infty}$
estimate for the vector-valued Hardy-Littlewood maximal function, the left-hand side of (\ref{nod}) can be bounded (up to a constant) by
$$ \|\chi\|_{L^\infty_u}\|M P_{k-j_*}f\|_{\mathbb L^p(\ell^2_j)}\les \|P_{k-j_*}f\|_{ L^p(\ell^2_j)},$$
 which establishes (\ref{nod}) through (\ref{ineq1}). Similarly, since \eqref{cha1} and $j\in S_{l,m}$,  
 we can also obtain   
\beq\label{nod-new1}
\|\h_{j,l}P_{k-j_*}f\|_{W_{r,p}(j\in S_{l,m})}\les 2^{(m+l)/r}\|f\|_{\mathbb H^p},
\eeq 
which is   useful in the proof of Lemma \ref{4l}. 
Indeed, the additional growth $2^{(l+m)/r}$ comes from the estimate of the  semi-norm $\|\cdot\|_{\dot{W}^{1,r}}^{1/r}$ on the right-hand side of (\ref{cha1}). 

 \subsection{Stationary phase analysis}
\label{sdef1}
Recall the expression (\ref{Adef}) of the phase $A_{j,l}^u(t,\xi)$.
We have
$$|\na_t A_{j,l}^u(t,\xi)|=u2^{j+l+j_*}\g'(2^{l+j_*}|t|)
t|t|^{-1}-2^{l+j_*}\xi.$$
Following (\ref{p8}), (\ref{curv}) and (\ref{smooth}) with $j=1$, along with (\ref{slim}), we deduce  that there exists a uniform  constant ${\mathcal{C}}$ such that
 for  $|\xi|\sim2^{k-j_*}$, $l\ge 1$,   $u\in[1,2]$ and  $j\in S_{l,m}$,
\vspace{2pt}

  (i)~$A_{j,l}^u(\cdot,\xi)$ does not have any critical points whenever $|k-m|> {\mathcal{C}}$, and
\vspace{2pt}

 (ii)~  $A_{j,l}^u(\cdot,\xi)$ may possess    a  critical point whenever $|k-m|\le {\mathcal{C}}$.
 \vspace{2pt}

For the case (i) that we call the non-degenerate case,  we will attain the desired estimate through standard integration by parts techniques. Specifically, we  establish the following lemma:
 \begin{lemma}\label{4l}
Let $n$, $l$, $m$ and $S_{l,m}$ be as in Proposition \ref{short2}.  Suppose that $p\in(1,\infty)$, $r\in[2,\infty)$, and the integer $k$ satisfies
 $|k-m|>{\mathcal{C}}$. Then there is $c_p>0$ such that 
\beq\label{Go4}
\|\h_{j,l}P_{k-j_*}f\|_{W_{r,p}(j\in S_{l,m})}\les_N 2^{-N(m+l)}
2^{-N(k+l){\ind {k\ge -l}}+c_p(k+l){\ind {k<-l}}} \|f\|_p
\eeq
for any  $N\in\N$.
\end{lemma}
 \begin{proof}[Proof of Lemma \ref{4l}]
By interpolation, \eqref{nod-new1} gives that it suffices to show \eqref{Go4} with  $p=r=2$.
Note that  $m+l>0$ and  $l\ge 1$. Using  $|\na_t A_{j,l}^u(t,\xi)|\gtrsim \max\{2^{m+l},2^{k+l}\}$ and the cancellation property $\int e^{i u2^j\g(2^{l+j_*}|t|)} \tilde \psi(t)dt=0$ (recall $\int_{{\mathbb S}^{n-1}}\Omega=0$), we have
$$|m_{k,j,l}(\xi)|\les_{N'} \min\{2^{-N'(m+l)},2^{-N'(k+l){\ind {k\ge -l}}}, 2^{-N'(m+l)+(k+l){\ind {k<-l}}}\}$$
for any sufficiently large $N'\in \N$. Then Plancherel's identity with Lemma \ref{l2.1} gives 
\beq\label{decay-q}
\|\h_{j,l}P_{k-j_*}f\|_{w_{2,2}(j\in S_{l,m})}\les_{N'} \min\{2^{-N'(m+l)},2^{-N'(k+l){\ind {k\ge -l}}}, 2^{-N'(m+l)+(k+l){\ind {k<-l}}}\}  \|f\|_{2}.
\eeq
In addition,  
since $\frac{\partial}{\partial u}\h_{j,l}P_{k-j_*}f$ can be written as 
$2^{m+l}{\mathcal T}_{j,l}P_{k-j_*}f$, where 
$\|{\mathcal T}_{j,l}P_{k-j_*}f\|_{w_{2,2}(j\in S_{l,m})}$ is bounded by the right-hand side of  \eqref{decay-q}, we can deduce  the  desired result by using \eqref{cha1}.  
\end{proof}
 Next, we consider the case (ii) called the degenerate case, which  is more involved.  To treat this case, we will use the method of stationary phase.  Let    $\gamma$ and  $F$  denote two functions
\beq\label{derf2}
\gamma(s):=(\g')^{-1}(s)\quad{\rm and}\quad F(s):=-s\gamma(s)+(\g\circ \gamma)(s).
\eeq
Since $(\g''\circ\gamma)(s) \gamma'(s)=1$,  we deduce 
\beq\label{derf}
F'(s)=-\gamma(s),\quad F''(s)=-\frac{1}{(\g''\circ\gamma)(s)}.
\eeq
The stationary phase theorem (see, for example,  \cite[page 360]{St93}) allows us to rewrite the multiplier $m_{k,j,l}$, defined by (\ref{mui0}),  as
\beq\label{stat}
m_{k,j,l}(\xi)=
\psi(\frac{\xi}{2^{k-j_*}})
\Big(e^{-i2^juF(|\xi|/(2^ju))} \Xi\big(2^{l+j_*}\xi,2^j\g(2^{l+j_*})u\big)+ \kappa\big(2^{j_*+l}\xi,2^j\g(2^{l+j_*})u \big)\Big),
\eeq
where the functions $\Xi:\R^n\times \R\to \R$ and $\kappa:\R^n\times \R\to \R$ satisfy
\begin{align}
|\p^{\A_1}_\xi\p^{\A_2}_u\Xi(\xi,u)|&\les_{\A}\hspace{5.5pt} \ \ (1+|(\xi,u)|)^{-n/2-|\A|}, \label{x1}\\
|\p^{\A_1}_\xi\p^{\A_2}_u\kappa(\xi,u)|&\les_{N,\A}\   (1+|(\xi,u)|)^{-N-|\A|}\label{x2}
\end{align}
for any $N\in\N$ and any multi-index $\A=(\A_1,\A_2)\in\Z_{\ge 0}^n\times \Z_{\ge 0}$.
In particular,  it follows that
  \beq\label{decy}
 |m_{k,j,l}(\xi)|\les 2^{-n(m+l)/2}\sim_n 2^{-n(k+l)/2}.
 \eeq

 We end this subsection by introducing
an elementary inequality, which
  will be frequently used to bound some error terms or  terms with a rapid decay.
For  $d\in\N$ and $N\ge n+1$, we define the function  $\E_{d,N}$ on $\R^n$ by
 \beq\label{bei}
\E_{d,N}(x):=\frac{2^{dn}}{1+|2^d x|^N}, \quad {\rm which\ satisfies }\  \|\E_{d,N}*f\|_{q_1}\les 2^{dn(1/q_2-1/q_1)}\|f\|_{q_2}
 \eeq
for all $ 1\le q_2\le q_1\le \infty$.
\subsection{A square-function estimate}
Keep in mind that the space $W_{r,p}(j\in S_{l,m})$ is defined by (\ref{nn}).
Lemma \ref{6l} below establishes a square-function estimate
based on the assumption that the corresponding single annulus estimate holds.
\begin{lemma}\label{6l}
Let $n$, $l$, $m$ and $S_{l,m}$ be as in Proposition \ref{short2}. Let $p_n\in [2,\infty)$ be a real number. Suppose that the integer $k$ satisfies
$|k-m|\le \mathcal{C}$ with  $\mathcal{C}$  given   as  in Subsection \ref{sdef1}. Assume   that 
 for each $p\in[p_n,\infty]$,  there exists a   constant $\mathcal{B}_{p,n,m,l}\gtrsim 2^{-n(m+l)/2}$ 
 such that for any $j\in S_{l,m}$,
 \beq\label{x01}
\|\h_{j,l}P_{k-j_*}f\|_{L^p(L^2_u)}\les \mathcal{B}_{p,n,m,l}\|f\|_p.
\eeq
Then, for each $p\in (p_n,\infty)$, we have
\beq\label{xxxx10}
\|\h_{j,l}P_{k-j_*}f\|_{W_{2,p}(j\in S_{l,m})}\les 2^{(m+l)/2}~\big(\mathcal{B}_{p,n,m,l} \ +\tilde{\mathcal{B}}\big) \|f\|_p,
\eeq
where
$\tilde{\mathcal{B}}=\min\{2^{(m+l)(n/p_n-n/p)}~\mathcal{B}_{p_n,n,m,l},
(m+l)\mathcal{B}_{p_n,n,m,l}^{p_n/p}
\mathcal{B}_{\infty,n,m,l}^{1-p_n/p}\}$.
 \end{lemma}
The proof of Lemma \ref{6l}, which is applied in the proof of Lemma \ref{51} and 
is deferred to Section \ref{sargu}, relies on the Fefferman-Stein sharp maximal estimate and a modification of Seeger's arguments as presented in \cite{Se88}. While this procedure is standard, the presence of a generic radial phase function adds a slight complexity to the proof. Additionally, we permit the bound without the loss of $2^{(m+l)\e}$ since the endpoint $p=p_n$ is omitted. 

\subsection{A technique of localization}\label{localargument}
For the operator $\h_{j,l}P_{k-j_*}$,  utilizing a spatial localization technique, we will deduce
 its $L^p\to L^p(L^2_u)$ estimate   from its associated $L^2\to L^p(L^2_u)$ estimate.
\begin{lemma}\label{local}
Let $n$, $l$, $m$, $S_{l,m}$ and $k$ be as in {Lemma} \ref{6l}.
Let
 $p\in [2,\infty]$ and   $j\in S_{l,m}$, and suppose that the  operator norm  $\|\h_{j,l}P_{k-j_*}\|_{L^2\to L^p(L^2_u)}$ is finite. Then the estimate
\beq\label{k01}
\|\h_{j,l}P_{k-j_*}\|_{L^p\to L^p(L^2_u)}\les_N 2^{(l+j_*)(n/2-n/p)}\|\h_{j,l}P_{k-j_*}\|_{L^2\to L^p(L^2_u)}+2^{-N(m+l)}
\eeq
holds for any $N\in\N$.
\end{lemma}

\begin{proof}[Proof of Lemma \ref{local}]
  For  ${\bf z}:=(z_1,\cdots,z_n)\in\Z^n$, denote by $Q_{\bf z}$   the cube $2^{l+j_*}\prod_{i=1}^n[z_i,z_i+1)$, and
denote  by $Q_{\bf z}^*$ the cube   centered at ${\bf z}$   with side-length $2^{l+j_*+11}$. By using  the notation   $f_{Q_{\bf z}}:=f1_{Q_{\bf z}}$ and recalling the expression (\ref{k11}),  we can write $\h_{j,l}P_{k-j_*}f(u,x)$ as
\beq\label{1te}
\begin{aligned}
\chi(u)\sum_{{\bf z}}(f_{Q_{\bf z}}*_xK_{k,j,l})(u,x)
&=\ \chi(u)\sum_{{\bf z}}(f_{Q_{\bf z}}*_x K_{k,j,l})(u,x)1_{Q_{\bf z}^*}(x)\\
&\ \ \ +\chi(u)\sum_{{\bf z}}(f_{Q_{\bf z}}*_x K_{k,j,l})(u,x)1_{(Q_{\bf z}^*)^c}(x)\\
&=:\ \mathcal{Y}_{1kjl}(u,x)+\mathcal{Y}_{2kjl}(u,x).
\end{aligned}
\eeq
Precisely, the choice of side lengths for $Q_{\bf z}$ and $Q_{\bf z}^*$ depends  on the estimate  (\ref{do1}).
To achieve (\ref{k01}), it  suffices to prove the inequality
\begin{align}
\|\mathcal{Y}_{2kjl}\|_{L^p(L^2_u)}
&\les_N\ 2^{-N(m+l)}\|f\|_p\label{aiq1}
\end{align}
for any $N\in\N$, and the estimate
\begin{align}
\|\mathcal{Y}_{1kjl}\|_{L^p(L^2_u)}
&\les\  2^{(l+j_*)(n/2-n/p)} \|\h_{j,l}P_{k-j_*}\|_{L^2\to L^p(L^2_u)}\ \|f\|_p.\label{aiq2}
\end{align}
Here the implicit constants in (\ref{aiq1}) and (\ref{aiq2}) are independent of $k,m,j,l$.
Using (\ref{do1}) and $|k-m|\le \mathcal{C}$, we immediately deduce  the pointwise inequality
\beq\label{b40}
|\mathcal{Y}_{2kjl}|(u,x)
\les_N 2^{-N(m+l)}\chi(u)|f|*\E_{k-j_*,N}(x).
\eeq
Then,
taking the $L^p(L^2_u)$ norm on both sides of (\ref{b40}), and using  the estimate (\ref{bei}),
we can achieve the desired   (\ref{aiq1}).
Next, we prove (\ref{aiq2}).
Note the estimate
$$
\begin{aligned}
\|\mathcal{Y}_{1kjl}\|_{L^p(L^2_u)}
\les&\  \big(\sum_{{\bf z}}\|\h_{j,l}P_{k-j_*}f_{Q_{\bf z}}\|_{L^p(L^2_u)}^p\big)^{1/p}\\
\les&\ \|\h_{j,l}P_{k-j_*}\|_{L^2\to L^p(L^2_u)}~(\sum_{\bf z}\|f_{Q_{\bf z}}\|_2^p)^{1/p}.
\end{aligned}
$$
Then (\ref{aiq2}) follows
from Fubini's theorem and
$\|f_{Q_{\bf z}}\|_2\les 2^{(l+j_*)(n/2-n/p)}\|f_{Q_{\bf z}}\|_p$.
\end{proof}
\subsection{Case 1: small $p$}
\label{sdef}
We shall prove Proposition \ref{short2} by applying  the following lemma:
\begin{lemma}\label{400l}
Let $n,l,m,k$ and $S_{l,m}$ be as in Lemma \ref{6l}. Suppose
$1<p\le 2$ and $2\le r\le p'$.  Then
\beq\label{Go3}
\|\h_{j,l}P_{k-j_*}f\|_{w_{r,p}(j\in S_{l,m})}\les 2^{- n(m+l)/p'} \|f\|_p.
\eeq
\end{lemma}
By using   (\ref{cha1}), we can  get from (\ref{Go3})  that if $n\ge 2$, $p\in (2n/(2n-1),2]$, $l,m,k$ and $S_{l,m}$ are given  as in Lemma \ref{6l},
\beq\label{big1}
\|\h_{j,l}P_{k-j_*}f\|_{W_{2,p}(j\in S_{l,m})}\les 2^{(m+l)(1/2-n/p')} \|f\|_p.
\eeq

\begin{proof}[Proof of Lemma \ref{400l}]
Interpolation shows that it suffices to prove the following estimates:
 \beq\label{71}
\|\h_{j,l}P_{k-j_*}f\|_{w_{r,1}(j\in S_{l,m})}\les  \|f\|_{\HH^{1}}\ \ (r=2,\infty),
\eeq
\beq\label{72}
\|\h_{j,l}P_{k-j_*}f\|_{w_{2,2}(j\in S_{l,m})}\les 2^{- n(l+m)/2} \|f\|_2.
\eeq
Since (\ref{71}) is a direct result of
 (\ref{nod}), we reduce the matter to   proving  (\ref{72}).
By Fubini's theorem,
$$\|\h_{j,l}P_{k-j_*}f\|_{w_{2,2}(j\in S_{l,m})}=
\big\|\|\h_{j,l}P_{k-j_*}f(u,x)\|_{L^2_u(L^2_x)}\big\|_{\ell^2(j\in S_{l,m})}.
$$
Then,
using  Plancherel's identity and (\ref{decy}), we  deduce
 $$
\|\h_{j,l}P_{k-j_*}f\|_{w_{2,2}(j\in S_{l,m})}
\les 2^{-n(m+l)/2}\| P_{k-j_*}f\|_{L^2(\ell^2_j)},
$$
which yields (\ref{72}) by employing inequality (\ref{ineq1}) in Lemma \ref{l2.1}.
\end{proof}

\begin{proof}[Proof of Proposition \ref{short2}]
Utilizing (\ref{cha1}) and 
  (\ref{Go3}),  we can infer that  the inequality
\beq\label{ms1}
\|\h_{j,l}P_{k-j_*}f\|_{W_{r,p}(j\in S_{l,m})}\les 2^{ (l+m)(1/r-n/p')} \|f\|_p
\eeq
holds whenever  $1<p\le 2$, $2\le r\le p'$ and $
|k-m|\le {\mathcal{C}}$.  Here $\mathcal{C}$ is  given   as  in Subsection \ref{sdef1}, and the growth $2^{(l+m)/r}$ comes from the estimate of the  semi-norm $\|\cdot\|_{\dot{W}^{1,r}}^{1/r}$ on the right-hand side of (\ref{cha1}). 
On the other hand, using Lemma \ref{4l}, we deduce  
\beq\label{ms2}
\|\h_{j,l}P_{k-j_*}f\|_{W_{r,p}(j\in S_{l,m})}\les 2^{-(m+l)}
2^{-(k+l){\ind {k\ge -l}}+c_p(k+l){\ind {k<-l}}} \|f\|_p
\eeq
holds whenever $1<p<\infty$, $2\le r<\infty$ and $|k-m|>{\mathcal{C}}$.

Finally,  we   achieve
 (\ref{Go2})
 by employing
 the Littlewood-Paley decomposition
 \beq\label{litt}
 f=\sum_{|k-m|\le \mathcal{C}}P_{k-j_*}f+\sum_{|k-m|>\mathcal{C}}P_{k-j_*}f,
 \eeq
  Minkowski's inequality, (\ref{ms1}), (\ref{ms2})  and $n/p'<1$.
\end{proof}
\subsection{Case 2: large $p$}\label{sub6.7}
In this subsection, we  give the proof of  Proposition \ref{short1} by using  the following lemma. Let $\tilde{p}_n=2(n+1)/{(n-1)}$.
\begin{lemma}\label{51}
 Let  $n,l,m,k$ and $S_{l,m}$ be as in Lemma \ref{6l}.
Suppose
 $ 2\le r< \infty$ and  $\tilde{p}_n< p< \infty$.  Then we have 
\beq\label{x10}
\|\h_{j,l}P_{k-j_*}f\|_{W_{r,p}(j\in S_{l,m})}\les 2^{- n(m+l)/p} \|f\|_p.
\eeq
\end{lemma}

\begin{remark}\label{amz}
For $n\ge 2$, we deduce by
interpolating (\ref{big1}) with (\ref{x10})  that there is an $\e>0$ such that  for every $p\in (2n/(2n-1),\infty)$,
$
\|\h_{j,l}P_{k-j_*}f\|_{W_{r,p}(j\in S_{l,m})}\les 2^{- (m+l)\e} \|f\|_p
$
with $l,m,k$ and $S_{l,m}$    as in Lemma \ref{6l}.
This together with (\ref{ms2}) and  Lemma \ref{l1} gives, for every $p\in (2n/(2n-1),\infty)$,
that  the short variation-norm estimate  $\big\|V_2^{\rm sh}({\bf H}f)\big\|_p\les \|f\|_p$ holds. On the other hand,
 the desired estimate for the long jump inequality was obtained in Section \ref{long} (see  (\ref{ka})).
 Thus, for  $n\ge 2$ and  every $p\in (2n/(2n-1),\infty)$, we can get  the jump inequality (\ref{jump100}).
\end{remark}
We first prove Proposition \ref{short1} under the assumption that Lemma \ref{51} holds.

\begin{proof}[Proof of Proposition \ref{short1}]
The proof is analogous to that of  Proposition \ref{short2}.
Using (\ref{litt}) and Minkowski's inequality, we   see  that  Proposition \ref{short1} is a direct result of
Lemma \ref{51} and (\ref{ms2}).
\end{proof}
It  remains  to show Lemma \ref{51}.
  \begin{lemma}\label{60}
  Let  $n,l,m,k$ and $S_{l,m}$  be as in Lemma \ref{51}.
  Suppose $\tilde{p}_n\le p\le \infty$.
 Then
 (\ref{x01}) with $\mathcal{B}_{p,n,m,l}=2^{-(m+l)(1/2+n/p)}$ holds.
 \end{lemma}
By accepting this lemma, we now prove   Lemma \ref{51}
 \begin{proof}[Proof of Lemma \ref{51}]
 It suffices to show the case $r=2$.
Clearly,  Lemma \ref{51} with $r=2$ follows from    Lemma \ref{60} and Lemma \ref{6l} with $n\ge 2$ and $p_n=\tilde{p}_n$.
 \end{proof}
 It  remains to prove Lemma \ref{60}. 
Instead of directly using the point-wise estimate (\ref{decy}), we shall use the asymptotic expansion (\ref{stat}) to provide a more detailed analysis. In what follows, for any function $h$ on $\R\times \R^n$, we use
$$
\F_{(2)}{h}(u,\xi) :=\int e^{-ix\cdot \xi}h(u,x)dx.
$$
Let $F$ denote the function defined by (\ref{derf2}),
and let   $G_a^\circ$, $G_a$ and $G_\kappa^\circ$  be three operators defined by
\beq\label{nota}
\begin{aligned}
\F_{(2)}{G_a^\circ f}(u,\xi):=&\ e^{-i2^juF(2^{-j}u^{-1}|\xi|)}\F_{(2)}{G_a f}(u,\xi),\\
\F_{(2)}{G_a f}(u,\xi):=&\ \chi(u) \Xi\big(2^{l+j_*}\xi,u2^j\g(2^{l+j_*})\big)\widehat{f}(\xi),\\
 \F_{(2)}{G_\kappa^\circ f}(u,\xi):=&\ \chi(u)  \kappa\big(2^{l+j_*}\xi,u2^j\g(2^{l+j_*})\big)\widehat{f}(\xi),
 \end{aligned}
 \eeq
 where the functions $\Xi$ and $\kappa$ satisfy  (\ref{x1}) and (\ref{x2}).
 Then we can rewrite  $\h_{j,l} P_{k-j_*}f$ as
 \beq\label{fen90}
 \h_{j,l}P_{k-j_*}f=G_a^\circ P_{k-j_*}f+G_\kappa^\circ P_{k-j_*}f.
 \eeq
This reduces the matter
 to proving the associated desired estimates for $G_a^\circ P_{k-j_*}f$ and $G_\kappa^\circ P_{k-j_*}f$.
\begin{proof}[Proof of Lemma \ref{60}]
Using  Lemma \ref{local}  we are able to reduce the matter to showing  
$$
\| \h_{j,l}P_{k-j_*} f\|_{L^p(L^2_u)}\les
2^{-(k-j_*)(\frac{n+1}{2}-
\frac{n}{p'})-\frac{l+j_*}{2}-\frac{n(m+l)}{2}}\|f\|_2,
\ \ \tilde{p}_n\le p\le \infty.
$$
By (\ref{fen90}), it is  sufficient to show that for each
$p\in [\tilde{p}_n,\infty]$,
\beq\label{step11}
\|G_\mathfrak{a}^\circ P_{k-j_*}f\|_{L^p(L^2_u)}\les
2^{-(k-j_*)(\frac{n+1}{2}-\frac{n}{p'})-\frac{l+j_*}{2}-\frac{n(m+l)}{2}} \|f\|_2,\ \mathfrak{a}\in\{a,\kappa\}.
\eeq
\vskip.1in
{\bf Proof of  (\ref{step11}) for $\mathfrak{a}=a$.} Let $\tilde{\chi}$  denote  the smooth function with compact support in $[2^{-3},2^3]$,
 and set $\tilde{\chi}=1$ on supp$\chi$.
By duality,
it suffices to show the following claim:  for all $ \mathcal{G}(u,x)\in L^{p'}(L^2_u)$,
\beq\label{se1}
\begin{aligned}
\|\psi(\frac{\xi}{2^{k-j_*}}) \int  e^{-i2^juF(2^{-j}u^{-1}|\xi|)} \tilde{\chi}(u)  \F_{(2)}{\mathcal{G}}(u,\xi)
 du \|_{L^2_\xi}
 \les
2^{-(k-j_*)(\frac{n+1}{2}-\frac{n}{p'})-\frac{l+j_*}{2}} \|\mathcal{G}\|_{L^{p'}(L^2_u)}.
 \end{aligned}
\eeq
Actually, using (\ref{nota}) we  write
$G_aP_{k-j_*}g=\tilde{K}_{j,k,l}*_xg$
with the kernel $\tilde{K}_{j,k,l}(u,x)$ described by
$$
\begin{aligned}
\tilde{K}_{j,k,l}(u,x)=&\ \chi(u)\int  e^{i\xi\cdot x} \psi(\frac{\xi}{2^{k-j_*}})~\Xi\big(2^{l+j_*}\xi,u2^j\g(2^{l+j_*})\big)~d\xi\\
=&\ 2^{n(k-j_*)}\chi(u)\int  e^{i2^{k-j_*}\xi\cdot x} \psi(\xi)~ \Xi\big(2^{l+k}\xi,u2^j\g(2^{l+j_*})\big)~d\xi.
\end{aligned}$$
By integration by parts  and (\ref{x1}), we have
$$|\tilde{K}_{j,k,l}(u,x)|+|\frac{\p \tilde{K}_{j,k,l}}{\p u}(u,x)|\les 2^{-n(m+l)/2}|(\chi,\chi')|(u) ~\E_{k-j_*,N}(x)
\quad {\rm for\ any}\ N\in\N,$$
which yields
from Sobolev's inequality that
\beq\label{f1}
\|\tilde{K}_{j,k,l}\|_{L^{1}(L^\infty_u)}\les \|\tilde{K}_{j,k,l}\|_{L^{1}(L^{1}_u)}+\|\frac{\p \tilde{K}_{j,k,l}}{\p u}\|_{L^{1}(L^{1}_u)}
\les 2^{-n(m+l)/2}.
\eeq
By
H\"older's inequality and Young's inequality,
we deduce from  (\ref{f1})  that
\beq\label{f2}
\|G_a  P_{k-j_*}g\|_{L^{p'}(L^2_u)}\les 2^{-n(m+l)/2} \|g\|_{L^{p'}(L^2_u)}.
\eeq
Then,   we finish the proof of   (\ref{step11})   with $\mathfrak{a}=a$  by  using (\ref{se1}) with $\mathcal{G}=G_a  P_{k-j_*}g$ and  (\ref{f2}).
\vskip.1in
It  remains to prove claim 
(\ref{se1}).
Let  $h:\R\times \R^n\to \R$ denote the function given by
$$
\F_{(2)}{h}(r,\eta):=\psi(\frac{r}{2^{k-j_*}})\int  e^{-i2^j u F(2^{-j}u^{-1}r)} \tilde{\chi}(u)    \F_{(2)}{\mathcal{G}}(u,r\eta)
 du.
 $$
Using polar coordinates we  rewrite   the square of  the left-hand side
 of (\ref{se1}) as
\beq\label{6.392}
\int_0^\infty \int_{\mathbb{S}^{n-1}}
|\F_{(2)}{h}(r,\theta)|^2d\sigma  r^{n-1} dr,
\eeq
where $\sigma$ represents the surface measure on $\mathbb S^{n-1}$. 
By the Stein-Tomas restriction theorem
$$ \int_{\mathbb{S}^{n-1}}|\widehat{f}(\theta)|^2d\sigma
\les \|f\|_{q'}^2\ {\rm\ whenever}\   \tilde{p}_n\le q\le \infty,$$
 the change of variable $x\to rx$, and Minkowski's inequality,
$$ (\ref{6.392}) \les\int_0^\infty \|h(r,\cdot)\|_{p'}^2 r^{n-1}dr
\les 2^{-(k-j_*)(n+1-2n/p')}
 \|\Theta\|_{p'}^2,
$$
where $\Theta(x):=(\int_0^\infty |r^n h(r,rx)|^2dr)^{1/2}$.
Hence,  it is sufficient for  (\ref{se1})   to show
\beq\label{g91}
\Theta^2(x)
\les\ 2^{-(l+j_*)} \|\mathcal{G}(\cdot,x)\|_2^2.
\eeq
Note that
$r^n h(r,rx)=\psi(\frac{r}{2^{k-j_*}})\int  e^{-i2^j u F(2^{-j}u^{-1}r)} \tilde{\chi}(u)   \mathcal{G}(u,x)
 du.$
Using the $TT^*$ argument, we rewrite $\Theta^2(x)$ as
\beq\label{suq}
\begin{aligned}
\Theta^2(x)
=&\ \int_0^\infty
\int_{u_1,u_2}|\psi(2^{-(k-j_*)}r)|^2
e^{iQ_{u_1,u_2}(r)} \tilde{\chi}(u_1)\tilde{\chi}(u_2)   \mathcal{G}(u_1,x)\overline{\mathcal{G}(u_2,x)}
du_1du_2 dr\\
=&\ 2^{k-j_*}\int_0^\infty
\int_{u_1,u_2}|\psi(r)|^2
e^{iQ_{u_1,u_2}(2^{k-j_*}r)} \tilde{\chi}(u_1)\tilde{\chi}(u_2) \mathcal{G}(u_1,x)\overline{\mathcal{G}(u_2,x)}
du_1du_2 dr.
\end{aligned}
\eeq
where the function $Q_{u_1,u_2}$ is given by
\beq\label{pha1}
Q_{u_1,u_2}(r):=2^ju_1F(2^{-j}u_1^{-1}r)
-2^ju_2F(2^{-j}u_2^{-1}r).
\eeq
It follows from  (\ref{derf2}) and (\ref{derf}) that
$
Q_{u_1,u_2}'(r)=\gamma(2^{-j}u_2^{-1}r)-\gamma(2^{-j}u_1^{-1}r),
$
which, according to the mean value theorem, leads to
$$
\begin{aligned}
    \frac{d}{dr}(Q_{u_1,u_2}(2^{k-j_*}r))
=&\ 2^{k-j_*} \big(\gamma(2^{k-j-j_*} u_2^{-1}r)-\gamma( 2^{k-j-j_*} u_1^{-1}r)\big)\\
=&\ 2^{k-j_*} \frac{2^{k-j-j_*} r(u_1^{-1}-u_2^{-1})}{(\g''\circ\gamma)
(2^{k-j-j_*} u_\circ r)}
\end{aligned}
$$
for some $u_\circ\in (u_1^{-1},u_2^{-1})$. We  {\bf claim}
\beq\label{claim1}
|\gamma
(2^{k-j_*-j} u_\circ r)|\sim  |\gamma
(2^{k-j_*-j})| \sim 2^{l+j_*}
\eeq
whenever $r\sim1$, and proceed with the proof of (\ref{g91}).
From  (\ref{claim1}),  (\ref{curv}) with $j=2$, and $u_1,u_2\sim1$ we deduce
\beq\label{aq1}
\begin{aligned}
   &\ \ |\frac{d}{dr}(Q_{u_1,u_2}(2^{k-j_*}r)) |\\
   =&\ 2^{k-j_*}  |\gamma
(2^{k-j_*-j} u_\circ r)||
   \frac{
2^{k-j_*-j} u_\circ r}{(\g''\circ\gamma)
(2^{k-j_*-j} u_\circ r) \gamma
(2^{k-j_*-j} u_\circ r)}||
   \frac{u_1^{-1}-u_2^{-1}}{u_\circ  }|\\
   \gtrsim&\  2^{k-j_*} 2^{l+j_*} |u_1-u_2|.
\end{aligned}
\eeq
Then, integration by parts with inequality 
$\sum_{m=2,3}|\frac{d^m}{dr^m}(Q_{u_1,u_2}(2^{k-j_*}r))|\les
|\frac{d}{dr}(Q_{u_1,u_2}(2^{k-j_*}r))|$
gives
$$\int_0^\infty |\psi(r)|^2 e^{iQ_{u_1,u_2}(2^{k-j_*}r)} dr
\les \frac{1}{1+2^{2(k+l)}|u_1-u_2|^2}.$$
Inserting this into (\ref{suq}), we have
$$
\begin{aligned}
\Theta^2(x)
\les&\ 2^{k-j_*}
\int_{u_1,u_2}
|\tilde{\chi}(u_1)\tilde{\chi}(u_2) | \frac{|\mathcal{G}(u_1,x){\mathcal{G}(u_2,x)}|}{1+2^{2(k+l)}|u_1-u_2|^2}
du_1du_2,
\end{aligned}
$$
which yields
from   H\"older's inequality that
\beq\label{c2}
 \begin{aligned}
  \Theta^2(x)
\les\ 2^{k-j_*} \|\mathcal{G}(\cdot,x)\|_2\big\| 
   \int |\mathcal{G}(u-u_1,x)|\frac{1}{1+2^{2(k+l)}|u_1|^2}d{u_1}\big\|_{L^2_u}.
 \end{aligned}
\eeq
 At last, we get (\ref{g91})  by applying  Young's inequality to (\ref{c2}).

To finish the proof of  (\ref{step11}) with $\mathfrak{a}=a$, it remains to show the {\bf claim}  (\ref{claim1}).
 Note that the first estimate is a direct result of Lemma \ref{lpro}.  For the second estimate, we need to exploit some useful information   from $j\in S_{l,m}$ defined  by (\ref{slim}). More precisely, we see that
$$2^k\sim 2^m\sim \frac{1}{2^{l}}\frac{\g(2^{l+j_*})}{\g(2^{j_*})}\quad {\rm whenever}\ j\in S_{l,m}.$$
By Lemma \ref{lpro} again, we deduce
\beq\label{r1}
|\gamma
(2^{k-j_*-j})|\sim |\gamma
\Big(\frac{1}{2^{l+j_*+j}}\frac{\g(2^{l+j_*})}{\g(2^{j_*})}\Big)|.
\eeq
Besides,  applying  (\ref{curv}) and (\ref{smooth}) with $j=1$, and (\ref{p9}), we deduce
\beq\label{r2}
\frac{\g(2^{l+j_*})}{2^{l+j_*+j}\g(2^{j_*})}
=\frac{\g(2^{l+j_*})}{2^{l+j_*}\g'(2^{l+j_*})}
\frac{\g'(2^{l+j_*})}{2^j\g(2^{j_*})}
\sim \g'(2^{l+j_*}).
\eeq
Finally,   the claim (\ref{claim1}) follows
 from  (\ref{r1}), (\ref{r2}),   (\ref{curv}) and (\ref{smooth}) with $j=1$, and (\ref{p9}).

{\bf Proof of  (\ref{step11}) for $\mathfrak{a}=\kappa$.}  We see from (\ref{x1}) and (\ref{x2}) that the function $\kappa$ decays faster  than the function $\Xi$,  therefore  the associated  proof of  (\ref{step11}) for $\mathfrak{a}=\kappa$  should be simpler. Write $G_\kappa^\circ P_{k-j_*} f:=f*_x \bar{K}_{j,k,l}$,
where the kernel  $\bar{K}_{j,k,l}(u,x)$ is given by
$$
\begin{aligned}
\bar{K}_{j,k,l}(u,x)=\ 2^{n(k-j_*)}\chi(u)\int  e^{i2^{k-j_*}\xi\cdot x} \psi(\xi) \kappa\big(2^{l+k}\xi,u2^j\g(2^{l+j_*})\big)d\xi.
\end{aligned}$$
Then, by following  the  arguments yielding the  estimate  for $\tilde{K}_{j,k,l}(u,x)$, we have
 $|\bar{K}_{j,k,l}(u,x)|\les_N 2^{-N(m+l)}\chi(u)\E_{k-j_*,N}(x)$ for any $N\in\N$,
and then
 $\|\bar{K}_{j,k,l}\|_{L^{q}(L^2_u)}\les_N 2^{-N(m+l)}2^{n(k-j_*)/q'}$ for any $q\ge 1$ and any $N\ge n+1$. This yields
 by Minkowski's inequality, Sobolev's inequality and (\ref{bei}) that
 $$\|G_\kappa^\circ P_{k-j_*} f\|_{L^p(L^2_u)}
 \les \|\bar{K}_{j,k,l}\|_{L^{2p/(p+2)}(L^2_u)}\|f\|_2
 \les_N 2^{-N(m+l)} 2^{(k-j_*)(n/p'-n/2)}\|f\|_2$$
 for any $N\ge n+1$. Taking $N$ large enough, we can end the proof of  (\ref{step11}) for $\mathfrak{a}=\kappa$.
\end{proof}

\begin{remark}\label{r1d}
For the one-dimensional case, it is necessary to establish the following  local smoothing estimate: There are two  constants $\e>0$ and $p_1\ge 4$ such that
for every $p\in (p_1,\infty)$,
$$
\|G_a^\circ P_{k-j_*}f\|_{W_{p,p}(j\in S_{l,m})}
\les 2^{-(m+l)\e}\|f\|_p,
$$
where $l$, $m$, $k$ and $S_{l,m}$  are as in Lemma \ref{51}.
For specific homogeneous phase functions, this estimate can be achieved  through Tao's bilinear estimate or decoupling theory.
However, achieving it without adding additional conditions is challenging due to the general nature of the phase function.
\end{remark}
\section{Proof of Lemma \ref{6l}}
\label{sargu}
In this section, we give the proof of  Lemma \ref{6l}.
\begin{proof}[Proof of Lemma \ref{6l}]
 By  (\ref{cha1}),
it suffices to show that for each $p\in (p_n,\infty)$,
\beq\label{x11}
\|\h_{j,l}P_{k-j_*}f\|_{w_{2,p}(j\in \mathcal{N}_{l,m})}\les
\big(\mathcal{B}_{p,n,m,l}+\tilde{\mathcal{B}}\big)\|f\|_p
\eeq
and
\beq\label{x11-20}
\|(\h_{j,l}P_{k-j_*}f)'\|_{w_{2,p}(j\in \mathcal{N}_{l,m})}\les 2^{m+l}
\big(\mathcal{B}_{p,n,m,l}+\tilde{\mathcal{B}}\big)\|f\|_p
\eeq
uniformly in $\mathcal{N}\in\N$,
where  $\mathcal{N}_{l,m}:=S_{l,m}\cap [-\mathcal{N},\mathcal{N}]$ and $(\h_{j,l}P_{k-j_*}f)'$ denotes $\frac{\partial}{\partial u}\h_{j,l}P_{k-j_*}f(u,x)$.  We only give the proof of \eqref{x11} since \eqref{x11-20} can be treated similarly.

For each $(l,m)$, we denote by $G_{l,m}$
the square function
\beq\label{fun1}
G_{l,m}(x):=(\sum_{j\in \mathcal{N}_{l,m}}\|\h_{j,l}P_{k-j_*}f(\cdot,x)\|_{2}^2)^{1/2}.
\eeq
Then, we rewrite  (\ref{x11}) as
\beq\label{x12}
\|G_{l,m}\|_p\les
\big(\mathcal{B}_{p,n,m,l}+\tilde{\mathcal{B}}\big)\|f\|_p.
\eeq
From (\ref{x01}) we first obtain that for all $p_n\le p\le \infty$,
\beq\label{zx1}
\|G_{l,m}\|_{p}\les \mathcal{N}^{1/2}\mathcal{B}_{p,n,m,l}\|f\|_p.
\eeq
To accomplish the goal (\ref{x12}), we need to remove the dependence of $\mathcal{N}$
on the right-hand side of (\ref{zx1}).  

For any function $h\in L^1_{loc}$, we
denote by $h^\#$ the Fefferman-Stein sharp maximal function of $h$, which can be  given by
$$
h^\#(x):=\sup_{ Q\ni x} \frac{1}{|Q|}\int_{Q}\left| h(y)-\frac{1}{|Q|}\int_Q h(w)dw\right|dy,$$
where the supremum is over all cubes $Q$ containing $x$.
   Since  (\ref{zx1}), we have  $\|G_{l,m}\|_q\les_q \|G^\#_{l,m}\|_q$ for all $q\in(p_n,\infty)$ (see \cite{FS72} for the details). Then,
 it is sufficient for (\ref{x12}) to show
\beq\label{x13}
\|G^\#_{l,m}\|_p\les \big(\mathcal{B}_{p,n,m,l}+\tilde{\mathcal{B}}\big)\|f\|_p,\quad
p\in (p_n,\infty).
\eeq
By triangle inequality, we have
$$
\begin{aligned}
 G^\#_{l,m}(x)\les&\  \sup_{Q\ni x}
\frac{1}{|Q|^{2}}\int_Q\int_Q
\Big(\sum_{j\in \mathcal{N}_{l,m}}\big\|\h_{j,l}P_{k-j_*}f(\cdot,w)
-\h_{j,l}P_{k-j_*}f(\cdot,y)\big\|_{2}^2\Big)^{1/2}dwdy.
\end{aligned}
$$
For every integer $L$, we
 denote by $\mathcal{Q}_L(x)$ the family of  cubes containing $x$ with  side-length in $(2^{L-1},2^L]$, and denote  two subsets of $\mathcal{N}_{l,m}$  by $\mathcal{N}_{ilm}^L$
$(i=1,2)$, which are given respectively by
$$
\begin{aligned}
\mathcal{N}_{1lm}^L:=\ \mathcal{N}_{l,m}\cap \{j\in\Z:\ l+j_* \le  L\}\quad{\rm and}\quad
\mathcal{N}_{2lm}^L:=\ \mathcal{N}_{l,m}\cap \{j\in\Z:\ l+j_*> L\}.
\end{aligned}$$
Define two maximal functions $A_{s,l,m}f(x)$ and $A_{b,l,m}f(x)$ by
$$
\begin{aligned}
A_{s,l,m}f(x):=&\ \sup_{L\in\Z}\sup_{Q\in \mathcal{Q}_L(x)}
\frac{1}{|Q|}\int_Q
\big(\sum_{j\in \mathcal{N}_{1lm}^L}\|\h_{j,l}P_{k-j_*}f(\cdot,y)\|_{2}^2~\big)^{1/2}dy,\\
A_{b,l,m}f(x):=&\ \sup_{L\in\Z}\sup_{Q\in \mathcal{Q}_L(x)}
\frac{1}{|Q|^{2}}\int_Q\int_Q\sum_{j\in \mathcal{N}_{2lm}^L}
\|\h_{j,l}P_{k-j_*}f(\cdot,w)-\h_{j,l}P_{k-j_*}f(\cdot,y)\|_{2}dwdy.
\end{aligned}
$$
Hence, $G^\#_{l,m}(x)\les  A_{s,l,m}f(x)+A_{b,l,m}f(x)$, and  it is sufficient for (\ref{x13}) to show
\begin{align}
\|A_{s,l,m}f\|_p\les&\ 2^{-n(m+l)/2}\ \|f\|_p\  \ \ \ \ \ {\rm for\ \ } 2\le p\le \infty,\label{fen1}\\
\|A_{b,l,m}f\|_p\les& \ (\mathcal{B}_{p,n,m,l}+\tilde{\mathcal{B}})\ \|f\|_p\ \ {\rm for\ \ }p_n< p< \infty.\label{fen2}
\end{align}

{\bf Proof of (\ref{fen1}).} By interpolation, it suffices to  prove (\ref{fen1}) for two endpoints  $p=2$ and $p=\infty$.
 Since $A_{s,l,m}f(x)\les MG_{l,m}(x)$, where  $G_{l,m}$ is given by (\ref{fun1}), we have
\beq\label{wp1}
\|A_{s,l,m}f\|_2
\les \|MG_{l,m}\|_2
\les (\sum_{j\in\Z}\|\h_{j,l}P_{k-j_*}f\|_{L^2_u(L^2)}^2)^{1/2}.
\eeq
From Plancherel's identity and (\ref{decy}) we can see
\beq\label{a00}
\|\h_{j,l}P_{k-j_*}f\|_{L^2_u(L^2)}
\les 2^{-n(m+l)/2}\|P_{k-j_*}f\|_2.
\eeq
Then, plugging (\ref{a00}) into (\ref{wp1}), and using (\ref{ineq1}) in Lemma \ref{l2.1}, we
obtain  (\ref{fen1}) for the case $p=2$.
Next, we consider the case $p=\infty$.
Note that  there exist an integer $L$ and a cube $Q_x$ containing $x$ with side-length in $(2^{L-1},2^L]$ such that
$$A_{s,l,m}f(x)\les\  \frac{1}{|Q_x|}\int_{Q_x}
(\sum_{j\in \mathcal{N}_{1lm}^L}\|\h_{j,l}P_{k-j_*}f(\cdot,y)\|_{2}^2)^{1/2}dy,$$
which yields by triangle inequality and H\"{o}lder's inequality that
\beq\label{112}
\begin{aligned}
A_{s,l,m}f(x)
\les&\ \ \Big(\frac{1}{|Q_x|}\int_{Q_x} \sum_{j\in \mathcal{N}_{1lm}^L}
\|\h_{j,l}P_{k-j_*}(f1_{B_x^L})(\cdot,y)\|_{2}^2dy\Big)^{1/2}\\
&+\frac{1}{|Q_x|}\int_{Q_x} \sum_{j\in \mathcal{N}_{1lm}^L}
\|\h_{j,l}P_{k-j_*}(f1_{(B_x^L)^c})(\cdot,y)\|_{2}dy\\
=:&\ R_{1lm}^Lf(x)+R_{2lm}^Lf(x),
\end{aligned}
\eeq
where 
$B_x^L$ is a ball containing $x$ with side-length $2^{L+20}$. The decomposition for $f$
with respect to  $B_x^L$ is motivated by (\ref{do1}).
 Using (\ref{a00}), we have
\beq\label{113}
\begin{aligned}
R_{1lm}^Lf(x)
\les&\  2^{-Ln/2}(\sum_{j\in\Z}\|\h_{j,l}P_{k-j_*}(f1_{B_x^L})\|_{L^2_u(L^2)}^2)^{1/2}\\
\les&\  2^{-Ln/2}2^{-n(m+l)/2}\|f1_{B_x^L}\|_2\\
\les&\  2^{-n(m+l)/2}\|f\|_\infty.
\end{aligned}\eeq
Thanks to (\ref{do1}), we can infer
$$
\begin{aligned}
R_{2lm}^L f(x)
\les_N&\  |Q_x|^{-1}
\int_{Q_x}\sum_{j\in \mathcal{N}_{1lm}^L}
\int_{\R^n\setminus B_x^L}\mathcal{E}_{k-j_*,N}(y-w)|f(w)|dwdy\\
\les_N&\ \sum_{j\in \mathcal{N}_{1lm}^L}\int_{|y|\ge 2^L}\mathcal{E}_{k-j_*,N}(y) dy~\|f\|_\infty.
\end{aligned}
$$
We may obtain  from   (\ref{sum1}) in Lemma \ref{l2.1}  with $\phi(j)=j_*$
that for any $\beta>0$,
\beq\label{us1}
\sum_{j:\ j_*\le L-l}
2^{\beta j_*}
\les_\beta
 2^{\beta(L-l)}.
 \eeq
Using (\ref{us1}), we further deduce by
setting $N>3n/2$ that
\beq\label{114}
\begin{aligned}
R_{2lm}^L f(x)
\les&\ \sum_{j\in \mathcal{N}_{1lm}^L} 2^{-(L+k-j_*)(N-n)}\|f\|_\infty\\
\les&\ 2^{-k(N-n)}2^{-L(N-n)}\sum_{j:j_*\le L-l}
2^{j_*(N-n)}\|f\|_\infty\\
\les&\ 2^{-(k+l)(N-n)}\|f\|_\infty.
\end{aligned}
\eeq
At last, (\ref{fen1}) for the case $p=\infty$ can be achieved by
plugging (\ref{113}) and (\ref{114}) into (\ref{112}).

{\bf Proof of (\ref{fen2}).} Let $n_0$ be a big positive integer such that $(j_1)_*\neq (j_2)_*$ whenever $|j_1-j_2|\ge n_0$.  Here $(j_{\rm i})_*$ ($\rm i=1,2$) is defined by
(\ref{brie}) with $\bar{k}=j_{\rm i}$.
Now,
we decompose $\mathcal{N}_{2lm}^L$ into $n_0$ sets $S_1$, $S_2, \cdots, S_{n_0}$, which are defined by  $$S_i:=\big(n_0\Z+i\big)\cap  \mathcal{N}_{2lm}^L,\quad i=1,\cdots,n_0.$$
Then, to prove (\ref{fen2}), it suffices to show that  for all $1\le i\le n_0$
\beq\label{fen3}
\|A_{b,l,m}^if\|_p\les \big(\mathcal{B}_{p,n,m,l}+\tilde{\mathcal{B}}\big)\|f\|_p,
\eeq
 where the function $A_{b,l,m}^if$ is given by
$$
A_{b,l,m}^if(x):=\ \sup_{L\in\Z}\sup_{Q\in \mathcal{Q}_L(x)}
\frac{1}{|Q|^{2}}\int_Q\int_Q\sum_{j\in S_i}
\|\h_{j,l}P_{k-j_*}f(\cdot,w)-\h_{j,l}P_{k-j_*}f(\cdot,y)\|_{2}dwdy.
$$
We only show \eqref{fen3} for the case $i=1$ since other cases $i=2,\ldots,n_0$ can be treated similarly. 

At this moment,  the function $v(j):=j_*+l-L$ is  an injection on $S_1$, which shows
that  $j=h(v)$,  the inverse function  of $v(j)$,
 is well-defined on the range of $v$  denoted by $\tilde{S}_1$. Changing the variable $j\to h(v)$
and
 splitting the set  $\tilde{S}_1$ into two subsets  $\tilde{S}_1'$ and $\tilde{S}_1''$, which are defined by
 $$\tilde{S}_1':=\tilde{S}_1\cap \{v\in\Z:\ v>k+l\},
 \quad\tilde{S}_1'':=\tilde{S}_1\cap \{v\in\Z:\ 1\le v\le k+l\},$$
we  can infer  the inequality
\beq\label{fen10}
A_{b,l,m}^1 f(x)
\le \sum_{v\in \tilde{S}_1'}A_{blmv}^1 f(x)
+\sum_{v\in \tilde{S}_1''}A_{blmv}^1 f(x)=:\mathcal{A}_{1lm}f(x)+\mathcal{A}_{2lm}f(x),
\eeq
where the operator $A_{blmv}^1$ is given  by
\beq\label{fd1}
A_{blmv}^1f(x):=
\sup_{L\in\Z}\sup_{Q\in \mathcal{Q}_L(x)}
\frac{1}{|Q|^{2}}\int_Q\int_Q\|\h_{v,l,k,L}f(\cdot,y)-\h_{v,l,k,L}f(\cdot,w)  \|_{2}dwdy
\eeq
in which
\beq\label{forma}
\h_{v,l,k,L}f(u,x):=\h_{h(v),l}P_{k-(v+L-l)}f(u,x).
\eeq
Thanks to (\ref{fen10}),  we reduce the matter   to showing 
\begin{align}
\|\mathcal{A}_{1lm}f\|_p\les&\  \mathcal{B}_{p,n,m,l}\|f\|_p,\quad p_n\le p\le \infty,\label{fen4}\\
\|\mathcal{A}_{2lm}f\|_p\les &\ \tilde{\mathcal{B}}\  \|f\|_p,\quad
\ \ \ \ \ \ \ p_n< p< \infty. \label{fen5}
\end{align}

First we prove (\ref{fen4}).
According to the fundamental theorem of calculus, we estimate the $L^2_u$ norm of  $\h_{v,l,k,L}f(u,y)$ minus $\h_{v,l,k,L}f(u,w)$  by the sum of  $O(n)$ terms like
\beq\label{aa1}
\begin{aligned} &\ 2^{k-(L-l+v)}|y-w|\int_0^1\|(\h_{h(v),l}\tilde{P}_{k-(v+L-l)}f)(\cdot,y+s(w-y))\|_2ds\\
\les&\ 2^{k-v+l}\int_0^1\|(\h_{h(v),l}\tilde{P}_{k-(v+L-l)}f)(\cdot,y+s(w-y))\|_2ds
\end{aligned}
\eeq
where  $|y-w|\les 2^L$ was applied. 
Write
$\tilde{\h}_{v,l,k,L}(u,x):=\h_{h(v),l}\tilde{P}_{k-(v+L-l)}(u,x).$
Thus, 
  we can bound $|A_{blmv}^1 f|$ by  the sum of  $O(n)$ terms, such as
 $2^{k-v+l} \sup_{L\in\Z} M(\|\tilde{\h}_{v,l,k,L} f\|_{L^2_u})$.
 Now, in order to show (\ref{fen4}), by $\ell^q\subset \ell^\infty$ for any $q<\infty$ and Fubini's theorem, it suffices to prove that for  $p_n\le p\le \infty$,
\beq\label{fen12}
\big(\sum_{L}\| \tilde{\h}_{v,l,k,L} f\|_{L^p(L^2_u)}^p\big)^{1/p}
\les \mathcal{B}_{p,n,m,l}\|f\|_p\quad {\rm whenever}\ v\in \tilde{S}_1'.
\eeq
Indeed,
changing the variable $j\to h(v)$ in (\ref{x01}) gives
\beq\label{variant}
\| \tilde{\h}_{v,l,k,L} f\|_{L^p(L^2_u)}\les \mathcal{B}_{{p,n,m,l}}\|\tilde{P}_{k-(v+L-l)}f \|_p,\ \ \ p_n\le p\le \infty,
\eeq
which immediately yields  (\ref{fen12})  by Lemma \ref{l2.1}.
\begin{remark}\label{r100}
In estimating  $\|\h_{v,l,k,L}f(\cdot,y)-\h_{v,l,k,L}f(\cdot,w)\|_2$, (\ref{aa1}) is inefficient if $v\le k+l$. In fact,  in cases of $v\le k+l$,  we can use a trivial inequality  $\|\h_{v,l,k,L}f(\cdot,y)-\h_{v,l,k,L}f(\cdot,w)\|_2
\le \|\h_{v,l,k,L}f(\cdot,y)\|_2+\|\h_{v,l,k,L}f(\cdot,w)\|_2$ instead. Consequently, we have
\beq\label{aa2}
|A_{blmv}^1f|\les\sup_{L\in\Z} M(\|\h_{v,l,k,L} f\|_{L^2_u}).
\eeq
\end{remark}

To achieve the desired (\ref{fen5}), it suffices to show
\begin{eqnarray}
\|A_{blmv}^1f\|_{p_n}
\les& \mathcal{B}_{{p_n,n,m,l}}&\|f\|_{p_n}, \label{aa3}\\
\|A_{blmv}^1f\|_{\infty}
\les& \min\{2^{nv/p_n}\mathcal{B}_{p_n,n,m,l},\ \mathcal{B}_{\infty,n,m,l}\}&\|f\|_{\infty}. \label{aa49}
\end{eqnarray}
Indeed, interpolating (\ref{aa3}) and (\ref{aa49}) gives
 $$
\|A_{blmv}^1f\|_p
\les \min\{2^{v(n/p_n-n/p)}\mathcal{B}_{p_n,n,m,l},\ \mathcal{B}_{p_n,n,m,l}^{p_n/p}~
\mathcal{B}_{\infty,n,m,l}^{1-p_n/p}\}\|f\|_p,
$$
which yields
$$\|\mathcal{A}_{2lm}f\|_p\les \min\Big\{\sum_{1\le v\le k+l} 2^{v(n/p_n-n/p)}
\mathcal{B}_{p_n,n,m,l},(k+l)\mathcal{B}_{p_n,n,m,l}^{p_n/p}~
\mathcal{B}_{\infty,n,m,l}^{1-p_n/p}\Big\}\|f\|_p.$$
Hence, using  $p>p_n$, we can achieve
(\ref{fen5}).
 
We first prove (\ref{aa3}).
Using (\ref{aa2})  we obtain
$
\|A_{blmv}^1 f\|_{p_n}^{p_n}
\les
\sum_{L\in\Z}\| {\h}_{v,l,k,L} f\|_{L^{p_n}(L^2_u)}^{p_n}.
$
Then (\ref{aa3}) follows from (\ref{variant}) with $p=p_n$ and Lemma \ref{l2.1}.
Next, we show (\ref{aa49}). 
Since
$
\|A_{blmv}^1f\|_{\infty}
\les\mathcal{B}_{\infty,n,m,l}\|f\|_{\infty}
$
derived from (\ref{variant}),  it suffices to show
\beq\label{aa4}
\|A_{blmv}^1 f\|_{\infty}
\les 2^{nv/p_n}\mathcal{B}_{p_n,n,m,l}\|f\|_{\infty}.
\eeq
We now bound the function $A_{blmv}^1 f(x)$.  
Decomposing
$f=f1_{B_x^{L+v+20}}+f1_{(B_x^{L+v+20})^c}=:F_{1Lv,x}+F_{2Lv,x}$,
where $B_x^{L}$   is defined as the statements  below (\ref{112}), we shall bound $A_{blmv}^iF_{dLv,x}$ ($d=1,2$).  This decomposition of $f$ is based on  (\ref{do1}).
By H\"older's inequality and (\ref{fen12}), the inequality
$$
\begin{aligned}
\frac{1}{|Q|}\int_Q
\|{\h}_{v,l,k,L} F_{{1Lv,x}}\|_{L^2_u}dy
\les&\
\big(\frac{1}{|Q|}\int_Q
\|{\h}_{v,l,k,L} F_{{1Lv,x}}\|_{L^2_u}^{p_n}dy\big)^{1/p_n}
\les\  2^{-Ln/p_n}\mathcal{B}_{p_n,n,m,l}
\|F_{{1Lv,x}}\|_{p_n}
\end{aligned}$$
holds
for all $Q\in \mathcal{Q}_L(x)$. This with  H\"older's inequality $\|F_{{1Lv,x}}\|_{p_n}\les |B_x^{L+v}|^{1/p_n}\|f\|_\infty$ implies
\beq\label{b1}
A_{blmv}^1 F_{1Lv,x}(x)
\les 2^{nv/p_n}\mathcal{B}_{p_n,n,m,l}\|f\|_\infty.
\eeq
In addition, since $1\le v\le k+l$ and $\mathcal{B}_{p,n,m,l}\gtrsim 2^{-n(m+l)/2}$,
if the inequality
\beq\label{b2}
A_{blmv}^1 F_{2Lv,x}(x)\les_N   2^{-N(m+l)}\|f\|_\infty,
\eeq
holds for any ${N}\in \N$,  we will  obtain  (\ref{aa4}) by combining (\ref{b1}) and  (\ref{b2}). Thus, it remains to show
(\ref{b2}).
Invoking  (\ref{forma}), for any $g\in L^\infty$, we can write
$\h_{v,l,k,L}g(u,x)=\chi(u)(g*_x\mathcal{K}_{klvL})(u,x)$, where
$$\mathcal{K}_{klvL}(u,x):=2^{n(k+l-v-L)}\int_{\R^{n+1}}\psi(\xi)
e^{i2^{k+l-v-L}\xi\cdot(x-2^{v+L}t)+i2^{h(v)}u\g(2^{v+L}t)}\tilde{\psi}(t)dtd\xi.$$
Integrating by parts in the variable $\xi$ shows $|\mathcal{K}_{klvL}(u,x)|\les_N 2^{-N(m+l)}\E_{k+l-v-L,N}(x)$ whenever $|x|\ge 2^{L+v+10}$. Then, for any
$u\sim 1$ and $y\in Q\in \mathcal{Q}_L(x)$, we deduce
$|\h_{v,l,k,L}F_{2x}(u,y)|\les_N 2^{-N(m+l)}\|f\|_\infty,$
 which, with the definition (\ref{fd1}) of the operator $A_{blmv}^1$,  yields
(\ref{b2}) by Young's inequality.
\end{proof}
\section{Sharpness of Theorem \ref{endpoint}}
\label{sharp}
In this section, we show that $r\ge p'/n$ is necessary.\footnote{The arguments also work for the case $n=1$.} To make the following construction of $\widehat{f_k}$ clearer, we show  the case  
$K(t)=\frac{t_m}{|t|^{n+1}}$ for some $1\le m\le n$.
Since we only employ certain conditions  related to  non-zero Gaussian curvature,
without loss of generality, from (\ref{curv}) and (\ref{smooth}) we may assume $\g''>0$ and   there are two constants $C_1\in (0,1/10)$ and
$C_2\in(1,\infty)$ such that
\beq\label{po1}C_1\le \frac{s\g''(s)}{\g'(s)}\le C_2.
\eeq
Recall the process in Subsection \ref{Fursub1}.  It follows from (\ref{po1}) that  for all  $l\in\Z$ and  $s\in (0,\infty)$,
\beq\label{po2}2^{C_1s}\le \frac{\g'(2^{l+s})}{\g'(2^l)} \le 2^{C_2s}.
\eeq
In addition,  we can also assume $\g(1)=1$ in what follows.
Let $k$ represent a sufficiently large integer, and define
\beq\label{para}
2^{\tilde{k}}:=\g'(2^k),\quad \quad \quad \widehat{f_k}(\xi):=\psi_\circ(2^{-\tilde{k}}\xi):=\psi(2^{-\tilde{k}}\xi_1)\cdots \psi(2^{-\tilde{k}}\xi_n)
\eeq
(here, $\psi$ is defined as in {\bf Notation}). Moreover, consider two fixed constants $\e_0$ and $\lambda$ that satisfy
\beq\label{choos}
0<C_2\e_0 \le 1/4,\quad \quad \quad \quad
C_1\lambda\ge 4.
\eeq
By the uncertain principle, (\ref{para}), and (\ref{curv})-(\ref{smooth}) with $j=1$,
 we obtain
$$\|f_k\|_p\sim 2^{\tilde{k}n/p'}\sim \big(2^{-k}\g(2^k)\big)^{n/p'}.$$
This implies that it suffices to show
\beq\label{go1}
\big\|V_r\big(H^{(u)}f_k:u\in[1,2]\big)\big\|_{L^p(I_{k,\e_0})}
\gtrsim \g(2^k)^{1/r}2^{-kn/p'},
\eeq
where the constant $\e_0$ is given by (\ref{choos}), and the set $I_{k,\e_0}$ is defined by
$$I_{k,\e_0}=\{x\in \R^n:\ 2^{k-2\e_0}\le |x|\le 2^{k-\e_0}\}.$$
Indeed, by taking $f=f_k$ in (\ref{var}) and setting $k\to \infty$ (so that $\g(2^k)\to \infty$),
  we immediately  deduce  $1/r\le n/p'$.

Next, we show (\ref{go1}). Let  $\Psi:\R\to [0,1]$ be a smooth function with compact support in $[2^{-\lambda},2^{\lambda}]$, where $\lambda$ is given by (\ref{choos}),
and set $\Psi$ such that $\sum_{j\in\Z}\Psi_j(|t|)=1$ for all $t\in \R^n\setminus\{0\}$, where $\Psi_j(\cdot):=\Psi(2^{-j}\cdot)$.
By  this partition of unity,  we can write $H^{(u)}f$ as
$$H^{(u)}f(x)=\sum_{j\in\Z}\int e^{i\xi\cdot x}
\psi_\circ(\frac{\xi}{2^{\tilde{k}}})~ \Xi_{k,j}^u(\xi) d\xi
=:\sum_{j\in\Z}
H^{(u)}_{j+k}f(x),$$
where
$ \Xi_{k,j}^u(\xi):=\int e^{-i\xi\cdot t+iu\g(|t|)}K(t)\Psi_{j+k}(t)dt$.
Using (\ref{p8}) and the assumption $\g(1)=1$, we have $2^{\A_0 k}\le \g(2^k)\le 2^{\A_1 k}$. This, combined with (\ref{para}), implies that  $\g(2^k)\sim 2^{\tilde{k}+k}$,
and $2^{N_0(k+\tilde{k})}\ge  2^{kn}$ whenever $N_0$ large enough.
Thus,
it is sufficient for (\ref{go1}) to show  that for any $N'\in\N$,
\begin{align}
\Big\|V_r\Big(\sum_{|j|\le \lambda+10}H^{(u)}_{j+k}f: u\in[1,2]\Big)\Big\|_{L^p(I_{k,\e_0})}
&\gtrsim\  \g(2^k)^{1/r}2^{-kn}2^{kn/p},\label{mai}\\
\Big\|V_r\Big(\sum_{|j|> \lambda+10}H^{(u)}_{j+k}f: u\in[1,2]\Big)\Big\|_{L^p(I_{k,\e_0})}
&\les_{N'}\  2^{-N'(k+\tilde{k})}2^{kn/p}.\label{eor}
\end{align}

We first
prove (\ref{eor}).  Using triangle inequality and (\ref{rou1}), we bound the
 $V_r$ semi-norm  on the left-hand side of  (\ref{eor}) by
$$
\begin{aligned}
 \sum_{|j|> \lambda+10}V_r\big(H^{(u)}_{j+k}f:u\in[1,2]\big)
\les& \sum_{|j|> \lambda+10}\|\p_u H^{(u)}_{j+k}f\|_{L^1_u([1,2])}.
\end{aligned}
$$
Then (\ref{eor}) follows from the estimate
\beq\label{ai01}
\|\p_u H^{(u)}_{j+k}f\|_{L^1_u([1,2])}
\les_{N} 2^{jb-N(\tilde{k}+k)}\min\{1,2^{-jN}\}
\eeq
for any $N\in\N$ and $|j|>\lambda+10$,
where the constant $b$ satisfies
\beq\label{fc2}
b=\left\{
\begin{aligned}
&\A_0 \ \ \ \ {\rm if} \ j\le -\lambda-10,\\
 & \A_1\  \ \ \  {\rm if}\  j\ge \lambda+10
\end{aligned}
\right.
\eeq
with $\A_0$ and $\A_1$ as in Subsection \ref{Fursub1}. It remains to prove  (\ref{ai01}).
Direct computation shows
\beq\label{dao1}\p_u H^{(u)}_{j+k}f(x)
=i2^{\tilde{k}n}\int \check{\psi_\circ}(2^{\tilde{k}}x-2^{\tilde{k}+k+j}t)
e^{iu\g(2^{k+j}|t|)}\g(2^{k+j}|t|)\tilde{\Psi}(t)dt,
\eeq
where $\tilde{\Psi}$ is a variant of ${\Psi}$, and its support is in  $[2^{-\lambda},2^{\lambda}]$.
 Since $\check{\psi_\circ}$ is a Schwartz function,
$|t|\in [2^{-\lambda},2^{\lambda}]$ and $x\in I_{k,\e_0}$, we
infer that for any $N_1\in\N$ and $|j|>\lambda+10$,
\beq\label{lowb}|\check{\psi_\circ}(2^{\tilde{k}}x-2^{\tilde{k}+k+j}t)|
\les_{N_1} \frac{1}{1+(\max\{|2^{\tilde{k}}x|,|2^{\tilde{k}+k+j}t|\})^{N_1}}
\les_{N_1} \frac{1}{1+2^{N_1(\tilde{k}+k)}\max\{1,2^{jN_1}\}}.
\eeq
Besides,
using (\ref{p8}), and (\ref{curv}) with $j=1$, we have
\beq\label{dao2}\g(2^{k+j}|t|)\sim \g(2^{k+j})=\frac{\g(2^{k+j})}{\g(2^{k})}\frac{\g(2^{k})}{2^k\g'(2^{k})}
2^k\g'(2^{k})\les 2^{jb}2^{k+\tilde{k}}
\eeq
with $b$ given by (\ref{fc2}).
Then,
combining (\ref{dao1}), (\ref{dao2}), and (\ref{lowb}) with $N_1\ge n+1$,  we can obtain
 (\ref{ai01}) by taking  $N=N_1-n$.

We now show (\ref{mai}).
Let  $\Phi(2^{-k}t):=\Phi_\la (2^{-k}t):=\sum_{|j|\le \lambda+10}\Psi_{j+k}(t)$. We can rewrite the sum over $j$ on the left-hand side of (\ref{mai}) as
\beq\label{lc1}
\begin{aligned}
 2^{\tilde{k}n}\int
\psi_\circ(\xi) e^{i2^{\tilde{k}}\xi\cdot x}
\Big(\int e^{-i2^{\tilde{k}+k}\xi\cdot t+iu\g_k(|t|)}K(t)\Phi(t)dt  \Big)d\xi,
\quad {\rm where}\ \g_k(\cdot):=\g(2^k\cdot).
\end{aligned}
\eeq
Let $\gamma_k$ be the inverse function of $(\g_k)'$,
and denote
$\tilde{F}(s):=-s\gamma_k(s)+(\g_k\circ\gamma_k)(s)$. Then we have  $\tilde{F}'=-\gamma_k$.
For every $\xi\in {\rm supp}\psi_\circ$ and each $u\in [1,2]$, we can observe from
$C_1\lambda\ge 4$ (see (\ref{choos})) that the phase    $2^{\tilde{k}+k}\xi\cdot t-u\g_k(|t|)$ in (\ref{lc1}) has a critical point $t_*=\xi|\xi|^{-1}\gamma_k(2^{\tilde{k}+k}u^{-1}|\xi|)\in \{2^{-\lambda}\le |t|\le 2^{\lambda}\}$. Then,
by  the stationary phase
theorem, we deduce from (\ref{po1}), (\ref{po2}) and (\ref{choos}) that the oscillatory integral  in the round bracket of (\ref{lc1}) can be described as
\beq\label{expa1}
e^{iu\tilde{F}(2^{\tilde{k}+k}|\xi|/u)}
\Xi_1\big(2^{\tilde{k}+k}\xi,\g(2^k)u\big)
+\kappa_1\big(2^{\tilde{k}+k}\xi,\g(2^k)u\big),
\eeq
where the functions  $\Xi_1$ and $\kappa_1$ satisfy (\ref{x1}) with $\Xi=\Xi_1$ and (\ref{x2}) with $\kappa=\kappa_1$, respectively. In particular, we  have
 $2^{n(\tilde{k}+k)/2}|\Xi_1|\ge c_0'$ for some uniform $c_0'>0$ (since $2^{n(\tilde{k}+k)/2}\Xi_1\big(2^{\tilde{k}+k}\xi,\g(2^k)u\big)\sim K(t_*)\Phi(t_*)$)   and 
$\|\kappa_1(2^{\tilde{k}+k}\xi,\g(2^k)u)\|_{C^1_u}\les_N 2^{-N({\tilde{k}+k})}$
for any $N\in\N$. Writing
$$H^{(u)}_{1,\circ}f(x):=2^{\tilde{k}n}\int
\psi_\circ(\xi) e^{i2^{\tilde{k}}\xi\cdot x}
\kappa_1(2^{\tilde{k}+k}\xi,\g(2^k)u) d\xi,$$
we can deduce by following the proof of the estimate (\ref{eor}) that for any $N\in\N$,
$$
\|V_r\big(H^{(u)}_{1,\circ}f: u\in[1,2]\big)\|_{L^p(I_{k,\e_0})}
\les_N 2^{-N({\tilde{k}+k})}2^{kn/p}.
$$
Thus, we focus on the contribution from the first
term in (\ref{expa1}).
Let $\Xi_2:=2^{n(\tilde{k}+k)/2}\Xi_1$
and define
$$H^{(u)}_{2,\circ}f(x):=2^{(\tilde{k}-k)n/2}\int
\psi_\circ(\xi) e^{i2^{\tilde{k}}\xi\cdot x+iu\tilde{F}(u^{-1}2^{\tilde{k}+k}|\xi|)}
\Xi_2(2^{\tilde{k}+k}\xi,\g(2^k)u)d\xi.
$$
We reduce the matter to estimating the operators $\{H^{(u)}_{2,\circ}\}_{u\in [1,2]}$.
Choose set $I_{k,\e_0}'\subsetneq I_{k,\e_0}$ with $|I_{k,\e_0}'|\sim 2^{kn}$ such that 
for  each $ x\in I_{k,\e_0}'$ and $u\in [1,2]$, by $0<C_2\e_0 \le 1/4$ in (\ref{choos}) and $\tilde{F}'=-\gamma_k$, the phase $2^{\tilde{k}}\xi\cdot x+u\tilde{F}(u^{-1}2^{\tilde{k}+k}|\xi|)$ has a critical point
$\xi_{u,x}= u2^{-\tilde{k}}\g'(|x|)|x|^{-1}x\in {\rm supp}\psi_\circ.$
By the stationary phase
theorem again,  we can rewrite  $H^{(u)}_{2,\circ}f$  as
$$
H^{(u)}_{2,\circ}f(x)=2^{-kn}e^{iu\g(|x|)}
\tilde{\Xi}_2\big(2^{\tilde{k}}x,\g(2^k)u\big)+O(2^{-N({\tilde{k}+k})})
$$
for any $N\in\N$,
where the function $\tilde{\Xi}_2$ is a variant of ${\Xi}_2$. Since the error term hidden in
 $O(2^{-N({\tilde{k}+k})})$ can be controlled as  $H^{(u)}_{1,\circ}f$, it suffices to prove
\beq\label{eds}
\|V_r\big(H^{(u)}_{3,\circ}f:u\in[1,2]\big)\|_{L^p(I_{k,\e_0}')}
\gtrsim \g(2^k)^{1/r}2^{-kn}2^{kn/p},
\eeq
where the operator $H^{(u)}_{3,\circ}$ is defined by
$$H^{(u)}_{3,\circ}f(x):=2^{-kn}e^{iu\g(|x|)}
\tilde{\Xi}_2\big(2^{\tilde{k}}x,\g(2^k)u\big).$$
Note $|\tilde{\Xi}_2(2^{\tilde{k}}x,\g(2^k)u)|\ge c_0$ for some uniform  $c_0>0$. For every $x\in I_{k,\e_0}'$, we denote
\beq\label{ed2}
\Lambda_x:=\{l\in\Z:\ {\pi}^{-1}{\g(|x|)}\le l\le 2{\pi}^{-1}{\g(|x|)}\},
\eeq
and then  obtain, for all $l\in \Lambda_x$, that
\beq\label{edd}
u_l^x:={\pi l}/{\g(|x|)}\in[1,2]\quad {\rm and}\quad |u_{l+1}^x-u_l^x|\les 1/{\g(|x|)}.
\eeq
From (\ref{ed2}) and (\ref{edd}) we can infer
\beq\label{end2}
\big(\sum_{l\in \Lambda_x}|e^{iu_{l+1}^x\g(|x|)}-e^{iu_{l}^x\g(|x|)}|^r\big)^{1/r}
\gtrsim \g(|x|)^{1/r}\ {\rm and}\ \big(\sum_{l\in \Lambda_x}|u_{l+1}^x-u_l^x|^r\big)^{1/r}\les \g(|x|)^{1/r-1}.
\eeq
Finally, using
triangle inequality, (\ref{end2}) and (\ref{p8}) we have
$$
\begin{aligned}
(\sum_{l\in \Lambda_x}|H^{(u_{l+1}^x)}_{3,\circ}f(x)-H^{(u_l^x)}_{3,\circ}f(x)|^r)^{1/r}
\gtrsim&\  2^{-kn}\g(|x|)^{1/r}\sim 2^{-kn}\g(2^k)^{1/r},
\end{aligned}
$$
which immediately yields (\ref{eds}) since $|I_{k,\e_0}'|\sim 2^{kn}$.
\section{Endpoint result for homogeneous  phase functions}
\label{weakht}
In this section, we show Theorem \ref{endpoint2} by applying
Bourgain's interpolation trick (see, e.g. Section 6.2 in \cite{CSWW99}) and using Lemma \ref{short22} below.
Let $\chi$ be the smooth function as defined in Subsection \ref{equ21}. For every $u\in \R$ and $l\in\Z$, let $H^{(u)}_l$ be the operator defined as in (\ref{sig}).
\begin{lemma}\label{short22}
 Let
 $n\ge2$, $\g\in \U$, $p'\in [2n,\infty)$ and let $r=p'/n$.
 Assume that $\g$ satisfies (\ref{homo}).
 Then
\begin{equation}\label{a2.2}
    \big\|\sum_{l\ge 1}\mathcal{G}_lf
    \|_{L^{p,\infty}}\lesssim
    \|f\|_{L^{p,1}}.
\end{equation}
where the function $\mathcal{G}_lf$ is defined by
\[
\mathcal{G}_lf(x):=\big(\sum_{j\in\Z}\|\chi(u)H^{(2^{j}u)}_{l+j_*}f(x) \|_{B_{r,1}^{1/r}(u\in \R)}^r\big)^{1/r},\ \ \ \ x\in\R^n.
\]
\end{lemma}
We first prove Theorem \ref{endpoint2} under the assumption that    Lemma \ref{short22} holds.
\begin{proof}[Proof of Theorem \ref{endpoint2}]
Following  the arguments in Section \ref{redut1}, we reduce the matter to  proving  the associated long and short variation-norm estimates.
Since
 Proposition \ref{long0} provides the desired long variation-norm estimate (the restriction $r=p'/n>2$ was used),  we
just need to prove the desired  short variation-norm estimate
\beq\label{dpp2}
\big\|V_{p'/n}^{\rm sh}(Hf)\big\|_{L^{p,\infty}}\les \|f\|_{L^{p,1}},\quad 2n\le p'<\infty.
\eeq
 Indeed,
employing  the embedding $B_{r,1}^{1/r}\hookrightarrow V^r$ for all $r\in[1,\infty)$,  (\ref{tc}), (\ref{edf}), we can obtain (\ref{dpp2}) by combining  Lemma \ref{l1}
and Lemma \ref{short22}. 
\end{proof}

\begin{proof}[Proof of Lemma \ref{short22}]
For each $l\ge 1$, recall the decomposition $\Z=\cup_{m\in\mathcal{P}_l}S_{l,m}$ with the sets $S_{l,m}$ and $\mathcal{P}_l$ defined as  in (\ref{slim}) and (\ref{Go5}), respectively.
Let  $a_l$ denote a constant given by $2^{a_l}:=2^{-l}\g(2^l)$. Since $\g$ satisfies (\ref{homo}), then, we have the estimate
$$2^{m+l}\sim \frac{\g(2^{l+j_*})}{\g(2^{j_*})}\sim \g(2^l)=2^{a_l+l}.
$$
 Using this estimate, we have a crucial  observation
\beq\label{ml1}
|m-a_l|\les 1\quad\quad ({\rm yielding}\quad \#\PP_l\les 1).
\eeq
Note that  the implicit constants are independent of $l$. 
It now  suffices to show that for each $1<p\le 2$ and each $2\le r\le p'$ (the range of $(p,r)$ can also be rewritten as $r\ge 2$ and $1<p\le r'$),
\beq\label{inter}
\big\|\mathcal{G}_{l,m}f
    \|_p\lesssim
    2^{(m+l)(1/r-n/p')}
    \|f\|_p,
\eeq
where the function $\mathcal{G}_{l,m}f$ is given by
$$\mathcal{G}_{l,m}f(x)=\big(\sum_{j\in S_{l,m}}\|\chi(u)H^{(2^{j}u)}_{l+j_*}f(x) \|_{B_{r,1}^{1/r}(u\in\R)}^r\big)^{1/r},\quad x\in\R^n.
$$
In fact, for each $r\geq 2$, considering that $n\geq 2$, we deduce from  (\ref{inter}) that there exist two constants  $p_1$ and $p_2$ such that
 $(nr)'\in (p_1,p_2)  \subset (1,r')$, and  $\|\mathcal{G}_{l,m}f
    \|_{p_i}\lesssim
    2^{(m+l)(1/r-n/p_i')}
    \|f\|_{p_i}$, $i=1,2$.
 Then, since
 (\ref{ml1}),
 the  desired result (\ref{a2.2}) follows by applying  Bourgain's interpolation trick to $\sum_{l\ge 1}\mathcal{G}_{l,m}$.

 At last, we show (\ref{inter}).
By the Littlewood-Paley decomposition (\ref{litt}),
(\ref{inter}) immediately  follows from
 (\ref{ms1}) and (\ref{ms2}). This completes the proof of Lemma \ref{short22}.
\end{proof}

\section*{Acknowledgements}
 The author thanks Mariusz Mirek and Jim Wright for discussions regarding discrete variational inequalities, which in turn yielded a concise proof of Lemma \ref{4l}.
The author thanks Jiecheng Chen and Meng Wang for discussions, and Shaoming Guo for helpful comments. Thanks also go to Haixia Yu for providing important references. 
This work was supported by the NSF of China 11901301.
Finally, the author  thanks  the anonymous referees for careful reading of the manuscript and constructive comments. 

\end{document}